\documentclass[a4paper,11pt]{amsart}
\usepackage{amsmath}
\usepackage{amsthm}
\usepackage{amssymb}
\usepackage{hyperref}
\usepackage{mathtools, fullpage, csquotes, microtype, lmodern}
\usepackage{mathrsfs}
\usepackage{longtable,tabularx}
\usepackage{graphicx}
\usepackage{fullpage}
\usepackage{float}
\usepackage{cleveref}
\usepackage{tikz}
\usepackage{pst-node}
\usepackage{tikz-cd}

\newcommand{\rvline}{\hspace*{-\arraycolsep}\vline\hspace*{-\arraycolsep}}
\newcommand{\C}{\mathbb{C}}

\newcommand{\kbalpha}{\overline{\mathrm{K}}_{\alpha}}

\usetikzlibrary{arrows}

\newtheorem{theorem}{Theorem}[section]
\newtheorem{proposition}[theorem]{Proposition}
\newtheorem{lemma}[theorem]{Lemma}
\newtheorem{corollary}[theorem]{Corollary}
\newtheorem{definition}[theorem]{Definition}
\newtheorem{example}[theorem]{Example}
\newtheorem{remark}[theorem]{Remark}
\numberwithin{equation}{section}

\title{The moment polytope of the abelian polygon space}
\author[N\, Daundkar]{Navnath Daundkar}
\address{Chennai Mathematical Institute, India}
\email{navnath@cmi.ac.in}
\author[P\, Deshpande]{Priyavrat Deshpande}
\address{Chennai Mathematical Institute, India}
\email{pdeshpande@cmi.ac.in}

\begin{document}
\begin{abstract}
The moduli space of $n$ chains in the plane with generic side lengths that terminate on a fixed line is a smooth, closed manifold of dimension $n-1$.
This manifold is also equipped with a locally standard action of $\mathbb{Z}_2^{n-1}$. 
The orbit space of this action is a simple polytope called the moment polytope. 
Interestingly, this manifold is also the fixed point set of an involution on a toric manifold known as the abelian polygon space. 
In this article we show that the moment polytope of the moduli space of chains is completely characterized by the combinatorial data, called the \emph{short code} of the length vector. 
We also classify aspherical chain spaces using a result of Davis, Januszkiewicz and Scott.  
\end{abstract}
\thanks{Both the authors are partially supported by a grant from the Infosys foundation. The project is also supported by the second author's MATRICS grant funded by the Science and Engineering Board, Government of India.}
\keywords{abelian polygon space, real toric variety, moment polytope, aspherical manifolds}
\subjclass[2010]{52B05, 52C25,  55R80, 58D28}

\maketitle

\section{Introduction}\label{intro}
A \emph{length vector}  $\alpha :=(\alpha_{1},\alpha_{2},\dots,\alpha_{m})$ is an $m$-tuple of positive real numbers. 
The \emph{spatial polygon space} $\mathrm{N}_{\alpha}$ parameterizes all piecewise linear paths in $\mathbb{R}^{3}$, whose side-lengths are prescribed by $\alpha$, up to rigid motions.
Formally, 
\[ \mathrm{N}_{\alpha}=\{(v_{1},v_{2},\dots,v_{m})\in (S^{2})^{m} : \displaystyle\sum_{i=1}^{m}\alpha_{i}v_{i} = 0 \}/\mathrm{SO}_{3},\]
where $\mathrm{SO}_{3}$ acts diagonally. 
If we choose $\alpha$ such that $\sum_{i=1}^{m}\pm \alpha_{i} \neq 0$ then the corresponding $\mathrm{N}_{\alpha}$ is a complex manifold of dimension $2(m-3)$. 
Such length vectors are called \emph{generic}. 
In this article we consider only generic length vectors.
The map $(x,y,z)\mapsto (x,y,-z)$ induces an involution on $\mathrm{N}_{\alpha}$. 
Let us denote the fixed point set of this involution by $\overline{\mathrm{M}}_{\alpha}$. 
It can also be written as: 
$$\overline{\mathrm{M}}_{\alpha}= \{(v_{1},v_{2},\dots,v_{m})\in (S^{1})^{m} : \displaystyle\sum_{i=1}^{m}\alpha_{i}v_{i} = 0 \}/\mathrm{O}_{2}.$$

The space $\overline{\mathrm{M}}_{\alpha}$ is also called the moduli space of planar polygons (or \emph{planar polygon space}); it is a smooth, closed manifold of dimension $m-3$. 

Polygon spaces (both spatial and planar) have been studied extensively in past few decades. 
For example, the integer cohomology ring of $\mathrm{N}_{\alpha}$ and the mod-$2$ cohomology ring of $\overline{\mathrm{M}}_{\alpha}$ computed by Hausmann and Knutson in \cite{JA}.  
In \cite{hkgrasmann}, the same authors described $\mathrm{N}_{\alpha}$ as a symplectic quotient of Grasmmannian $G(2,\mathbb{C}^{m})$ of $2$-planes in $\mathbb{C}^{m}$. 

A $2n$-dimensional symplectic, smooth manifold is called a toric manifold if it admits a Hamiltonian action of an $n$-dimensional torus.  
Note that a toric manifold is equivariantly diffeomorphic to a non-singular toric variety (see \cite[Theorem 7.5]{MT})
It turns out that for some choices of $\alpha$ the manifold $\mathrm{N}_{\alpha}$ is a toric manifold.
The half-dimensional torus action is given by bending flows \cite{KM2}. 
Hausmann and Roudrigue \cite[Proposition 6.8]{JE} provided a (combinatorial) sufficient condition for $\mathrm{N}_{\alpha}$ to be a toric manifold. 
 
A smooth $2n$-dimensional (respectively $n$-dimensional) manifold $M$ is called a quasi-toric (respectively small cover) if it has a locally standard
action of $n$-dimensional torus (respectively $\mathbb{Z}_{2}^{n}$) such that the orbit space can be identiﬁed with a simple $n$-polytope. 
These manifolds were introduced by Davis and Januszkiewicz in \cite{MT}.  
They showed that many topological invariants of these spaces are encoded in the combinatorics of the associated quotient polytope.  
In the same paper they prove that the small cover can be realized as a fixed point set of an involution on a quasi-toric manifold. 
Note that in this article we will refer to the quotient polytope as the moment polytope even in the context of small covers. 

Since the Hamiltonion action of a torus is locally standard (see \cite[Section 7.3]{MT} for a proof) and the image of the moment map is a simple convex polytope, hence toric manifolds are quasi-toric. 
Moreover, the fixed point set of an anti-symplectic involution on a toric manifold is a small cover. 
This is true since the image under the restriction of the moment map to the fixed point set is again the same moment polytope. 
It can be seen that $\overline{\mathrm{M}}_{\alpha}$ is the fixed point set of an anti-symplectic involution on $\mathrm{N}_{\alpha}$.
Therefore it has the structure of a small cover whenever $\mathrm{N}_{\alpha}$ is a toric manifold.
We refer the interested reader to the paper of Hausmann and Knutson \cite{hkgrasmann} for the terminologies related to symplectic structure that are not defined here. 

Whether or not a small cover is aspherical is also determined by the moment polytope, as given by the following result of Davis, Januszkiewicz and Scott \cite[Theorem 2.2.5]{MTR}. 

\begin{theorem}\label{artv}
Let $M$ be a small cover and $P$ be the associated quotient polytope. 
Then the following statements are equivalent: 
\begin{enumerate}
\item $M$ is aspherical.
\item The boundary complex of $P$ is dual to a flag complex.
\item The dual cubical subdivision of $M$ is nonpositively curved.
\end{enumerate}
\end{theorem}

Our focus in this paper is to determine which planar polygonal spaces are aspherical; in order to do that we consider the polygonal spaces that are small covers.
In fact, we describe the face poset of the associated moment polytope in terms of the length vector information.

Some of the planar polygonal spaces that are small covers can also be described as moduli spaces of chains.

\begin{definition}\label{dcs}(Chain space)
Let $\alpha=(\alpha_{1},\alpha_{2},\dots,\alpha_{m})$ be a length vector. The chain space corresponding to $\alpha$ is defined as : 
\[
\mathrm{Ch}(\alpha)=\{(v_{1},v_{2},\dots,v_{m-1})\in (S^{1})^{m-1} : \displaystyle\sum_{i=1}^{m-1}\alpha_{i}x_{i} = \alpha_{m} \}/\mathbb{Z}_{2},
\]
where $v_{i}=(x_{i},y_{i})$ and the group $\mathbb{Z}_{2}$ acts diagonally.
\end{definition}

Note that $\mathrm{Ch}(\alpha)\neq \emptyset$ if $\alpha_{m}\leq \sum_{i=1}^{m-1}\alpha_{m-1}$. 
In fact, if $\alpha$ is generic then $\mathrm{Ch}(\alpha)$ is a smooth, closed manifold of dimension $m-2$.
The elements of a chain space can be thought of as a piece-wise linear path with side lengths $\alpha_{1},\alpha_{2},\dots,$ $\alpha_{m-1}$ terminating at the line $x=\alpha_m$, considered up to reflection across the $X$-axis.

The spatial version of the chain space for a generic length vector was introduced by Hausmann and Knutson; it was referred to as the \emph{abelian polygon space}. 
It is a \emph{toric manifold} (see \cite[Section 1]{JA} for a proof). 
An element of this spatial version can be viewed as a piece-wise linear path with side lengths $\alpha_{1},\alpha_{2},\dots,\alpha_{m-1}$ terminating on the plane $x=\alpha_{m}$, modulo the rotations about the $X$-axis. 
It is easy to see that $\mathrm{Ch}(\alpha)$ is a fixed point set of an anti-symplectic involution on its spatial version. 
Thus $\mathrm{Ch}(\alpha)$ is a small cover. 

In this article we characterize those $\alpha$ for which $\mathrm{Ch}(\alpha)$ is aspherical. 
In order to do that we introduce a combinatorial object called the \emph{short code} of a length vector (see, \Cref{tscmglv}). 
We first show that the short code of a (generic) length vector completely determines the moment polytope of the associated chain space.  
Then we establish a combinatorial condition to determine whether or not the chain space is aspherical. 
In particular, we prove the following : 

\begin{theorem}\label{main}
Let $\alpha$ be a generic length vector. Then the corresponding chain space is aspherical  if and only if the short code of $\alpha$ is one of the following :
\begin{enumerate}
\item $<\{1,2,\dots,m-3,m\}>$,
\item $<\{1,2,\dots,m-2,m\}>$,
\item $<\{1,2,\dots,m-4,m-2,m\}>$,
\item $<\{1,2,\dots,m-4,m-1,m\}>$.
\end{enumerate}
\end{theorem}

The paper is organized as follows : In \Cref{pre}, 
we first state the motivation for the results in this paper. 
We also explain why it is difficult to investigate whether or not of general planar polygon space is aspherical using techniques of Davis-Januszkiewicz-Scott.
In \Cref{tsc}, we introduce the notion of \emph{short codes}. 
We then show that a chain space is diffeomorphic to some planar polygon space. 
In \Cref{fpmp}, we describe the moment polytope of a chain space and explicitly describe it in some special cases. 
Next we determine its number of facets in terms of the short subset information. 
We then introduce the poset of \emph{admissible subsets} corresponding to a generic length vector and  prove that it is isomorphic to the face poset of the corresponding moment polytope. 
In \Cref{pmt}, we prove our main result \Cref{main}.

\section{Motivation}\label{pre}
The results in our paper are motivated by the techniques developed by Davis, Januszkiewicz and Scott  to conclude that the (real part of) moduli space of certain point configurations is aspherical. 
In particular, the authors consider the following situation: $M_{\C}$ is an $n$-dimensional complex manifold and $D_{\C}$ is a smooth divisor. 
Assume that there is a smooth involution defined on $M_{\C}$ which is locally isomorphic to complex conjugation on $\C^n$.
This situation provides a `real version' of the pair $(M_{\C}, D_{\C})$, which we denote by $(M, D)$, where $M$ is the fixed point set of the involution and $D = D_{\C}\cap M$.
In such a case $D$ is a union of codimension-one smooth submanifolds which is locally isomorphic to an arrangement of hyperplanes. 
Consequently, the complement of $D$ is a disjoint union of cells, called \emph{chambers}, which are combinatorially equivalent to simple convex polytopes. 
The cell structure induced by $D$ has a cubical subdivision (i.e., one can subdivide the polytopal cells to obtain a tiling by cubes).
For (smooth) manifolds equipped such a cell structure the authors established a combinatorial condition to check whether or not the manifold is aspherical. 
In the remaining of this section we formally state their results that are useful to us and also show that the planar polygonal spaces posses similar structure. 

\begin{definition}\label{flagtope}
An $n$-dimensional convex polytope is called \emph{simple} if each vertex is an intersection of $n$ codimension-$1$ faces (also called  \emph{facets}).  
A simple polytope is called a flagtope 
if every collection of its pairwise intersecting facets has a nonempty intersection.   
\end{definition}

Flagtopes have many interesting combinatorial properties. 
For example, an $n$-dimensional flagtope has at least $2n$ facets. 
In fact, an $n$-cube is the only flagtope with the minimum possible facets. 

\begin{definition}\label{simplesc}
If all the cells of a regular cell complex are combinatorially equivalent to a simple polytope then it is called a simple cell complex.
\end{definition}

An interesting property of simple cell complexes is that it is possible to subdivide each cell into cubes, thus turning it into a cubical complex. 
Recall that a zonotope is a polytope all of whose faces are centrally symmetric (see \cite[Chapter 7]{G} for more details). 

\begin{definition}\label{zonotopal}
A regular cell complex is called zonotopal if every cell is combinatorially isomorphic to a zonotope. 
\end{definition}

We can now state an important consequence of Gromov's lemma that provides a combinatorial condition to check whether the given cubical complex is non-positively curved or not. 
The details of the following result can be found in \cite[Section 1.6]{MTR}

\begin{theorem}\label{Dav}
Suppose that $\mathrm{K}$ is a simple cell complex structure on an $n$-dimensional smooth manifold such that its dual cell complex is zonotopal.
If each $n$-cell $P$ of $K$ is combinatorially isomorphic to a flagtope then $K$ is aspherical. 
\end{theorem}

We now show that the planar polygon spaces have a natural simple cell structure (such that the dual is zonotopal). 
However, it will turn out that in most cases the top-cells are not flagtopes. 

\begin{definition} \label{opps}
Let $\alpha=(\alpha_{1},\alpha_{2},\dots,\alpha_{m})$ be a generic length vector. The moduli space $\mathrm{M}_{\alpha}$, of planar polygons is a collection of all closed piece-wise linear paths in the plane up to orientation preserving isometries with side lengths $\alpha_{1}, \alpha_{2},\dots, \alpha_{m}$, i.e., 
$$\mathrm{M}_{\alpha} = \{(v_{1},v_{2},\dots,v_{m})\in (S^{1})^{m} : \displaystyle\sum_{i=1}^{m}\alpha_{i}v_{i} = 0 \}/\mathrm{SO}_{2}.$$ 
\end{definition}

It is known that (see, \cite[Theorem 1.3]{FAR}) the moduli space $\mathrm{M}_{\alpha}$ is a closed, smooth manifold of dimension $m-3$.
The planar polygon space $\mathrm{M}_{\alpha}$ admits an involution $\tau$ defined by
\begin{equation}\label{invo}
    \tau(v_{1},v_{2},\dots,v_{m})=(\bar{v}_{1},\bar{v}_{2},\dots,\bar{v}_{m}),
\end{equation}
where $\bar{v}_{i}=(x_{i},-y_{i})$ and $v_{i}=(x_{i},y_{i})$. Geometrically, $\tau$ maps a polygon to its reflected image across the $X$-axis. 
Since the length vector is generic, the involution $\tau$ does not have fixed points. 
It is easy to see that  $\mathrm{M}_{\alpha}$ is a double cover of $\overline{\mathrm{M}}_{\alpha}$. 

We denote the set $ \{1,\dots, m\}$ by $[m]$. 
An ordered partition of $[m]$ into $k$-blocks is a tuple $(J_{1},J_{2},\dots,J_{k})$, where $J_{i}$'s are pairwise disjoint subsets of $[m]$ whose union is $[m]$. 
The set of all ordered partitions of $[m]$ forms a poset under the refinement order relation. We consider now a special class of partitions, called a cyclically ordered partition, which will enable us to define a regular cell structure on $\mathrm{M}_{\alpha}$.

\begin{definition}\label{cop} A cyclically ordered partition of $[m]$ is an ordered partition $(J_{1},J_{2},\dots,J_{k})$ which is equivalent
to any ordered partition obtained from it by a cyclic permutation of its blocks, i.e.,  $(J_{1},J_{2},\dots,J_{k})$,  $(J_{2},J_{3},\dots,J_{k},J_{1})$,   $\dots,  (J_{k},J_{1},\dots,J_{k-1})$ are all equivalent. 
\end{definition}

\begin{definition}\label{adp}
Given a length vector $\alpha$ a subset $I\subset [m]$ is called  $\alpha$-\emph{short} if 
\[\sum_{i\in I} \alpha_i  < \sum_{j \not \in I} \alpha_j.\] 

A cyclically ordered partition $(J_{1},J_{2},\dots,J_{k})$ of the set $[m]$ is said to be $\alpha$-admissible if $J_{i}$ is $\alpha$-short for all $1\leq i\leq k$.	
\end{definition}

Panina \cite{GP} described a regular cell structure on $\mathrm{M}_{\alpha}$. 
The $k$-cells of this complex correspond to $\alpha$-admissible partitions of $[m]$ into $k+3$ blocks. 
The boundary relations on the cells are described by the partition refinement. 
Let us denote this cell structure by $\mathrm{K}_{\alpha}$.

It is proved \cite[Proposition 2.12]{GP} that the complex $\mathrm{K}_{\alpha}$ is a simple cell complex.



\begin{proposition}
 For a generic $\alpha$ the dual of $\mathrm{K}_{\alpha}$ is zonotopal. 
\end{proposition}

\begin{proof}
We refer the reader to \cite[Section 2]{GP} where it is proved that each dual cell is a product of permutohedra (it is a particular type of simple zonotope). 
\end{proof}

\begin{remark}\label{ppsscc}
 It is clear that the involution $\tau$ (\Cref{invo}) defined on $\mathrm{M}_{\alpha}$ is cellular; the cell labeled by $(I_{1},I_{2},\dots,I_{k})$ mapped to the cell labeled by $(I_{k},I_{k-1},\dots,I_{2},I_{1})$.
 Therefore, $\overline{\mathrm{M}}_{\alpha}$ has a simple cell structure with cells labeled by bi-cyclically ordered $\alpha$-admissible partitions. 
 We denote this cell structure on $\overline{\mathrm{M}}_{\alpha}$ by $\overline{\mathrm{K}}_{\alpha}$.  
 Note that the dual of this cell complex is zonotopal. 
\end{remark}

We now know that for a generic
$\alpha$ the complex $\kbalpha$ satisfies the premise of \Cref{Dav}.
So we find out whether or not the top-cells of $\kbalpha$ are flagtopes. 
We do this analysis by considering the number of sides. 
If $m=3$ then there is only one possibility, $\overline{\mathrm{K}}_{\alpha}$ is always a point. 
If $m=4$ then again $\overline{\mathrm{K}}_{\alpha}\cong S^{1}$ for any generic $\alpha$. 
If $m=5$ then $\overline{\mathrm{K}}_{\alpha}$ is either a torus or the nonorientable surfaces of genus $1,2,3,4$ and $5$. 
It is not difficult to verify that in case of genus $5$ all of the $12$ top-cells are pentagonal, hence they are flagtopes. 
However, in all other cases the cell structure contains at least one triangular top-cell, see \cite[Section 2]{GP} for details. 
So we can't appeal to \Cref{Dav} in this case. 
In fact, for the same reason we can't use this theorem for the larger values of $m$.

\begin{proposition}\label{thmfails}
Let $m\geq 6$ and $\alpha$ be a generic length vector with $m$ components. 
Then, the cell structure $\kbalpha$ contains no top-cell isomorphic to a flagtope. 
\end{proposition}

\begin{proof}
It is straightforward to see that any top-dimensional (i.e. $(m-3)$-dimensional) cell of $\kbalpha$ has at most $m$ facets. 
Therefore, the only possible values of $m$ for which the number of facets are greater than or equal to $2(m-3)$ are $m=3,4,5$ and $6$. 
Therefore, if $m\geq 7$ then none of the top-dimensional cells of $\overline{\mathrm{K}}_{\alpha}$ can be combinatorially isomorphic to a flagtope.
The cases $m=3,4,5$ are dealt above, so here we deal with the case $m = 6$, where the top-dimensional cells of $\kbalpha$ may have $6$ facets.

From \cite[Theorem 3]{GR} we know that there is only one $3$-dimensional flagtope with $6$ facets and that is a $3$-cube.
We now show that it is impossible for the corresponding $\kbalpha$ to have each top-dimensional cell of $\overline{\mathrm{K}}_{\alpha}$ isomorphic to a $3$-cube.
To see this assume the contrary that a cell denoted by $\sigma=(1,2,3,4,5,6)$ is combinatorially isomorphic to a $3$-cube. 
We write $\{i,j\}$ as $ij$ for short. 
Then $\sigma$ have following $6$ facets, $$(12,3,4,5,6), (1,23,4,5,6), (1,2,34,5,6), (1,2,3,45,6), (1,2,3,4,56), (2,3,4,5,16)$$ which are $2$-cubes. 
Consider the following $2$-faces and their edges. 
\begin{enumerate}
\item  The possible faces of $(2,3,4,5,16)$ : $$(23,4,5,16), (2,34,5,16) ,(2,3,45,16), (2,3,4,156) \hspace{.1cm} \text{or} \hspace{.1cm} (3,4,5,126)$$

\item The possible faces of $(1,2,34,5,6)$ : $$(12,34,5,6), (1,2,34,56), (2,34,5,16), (1,234,5,6) \hspace{.1cm} \text{or} \hspace{.1cm} (1,2,345,6)$$

\item The possible faces of $(1,2,3,45,6)$  : $$(12,3,45,6), (1,23,45,6), (2,3,45,16), (1,2,345,6) \hspace{.1cm} \text{or} \hspace{.1cm} (1,2,3,456)$$
\end{enumerate}

Note that $156$ cannot be short. 
Otherwise both $(1,2,34,5,6)$ and $(1,2,3,45,6)$ will be isomorphic to $2$-simplex, because then $234$, $345$ and $456$ will be long subsets. 
Since we assumed $(2,3,4,5,16)$ is a $2$-cube, $126$ must be short. But then a $2$-face $(1,2,3,45,6)$ will be isomorphic to a $2$-simplex. 
Which is contradiction. 
Therefore, if $m=6$, it is impossible to have each top-dimensional cell of $\overline{\mathrm{K}}_{\alpha}$ isomorphic to a $3$-cube. 
\end{proof}

Although the natural cell structure of $\kbalpha$ is simple we cannot apply the Davis-Januszkiewicz-Scott schema as the top-cells are not flagtopes. 
However, some of these polygon spaces are small covers and the Davis-Januszkiewicz-Scott theorem can be applied to this situation since they have a natural cell structure tiled by the moment polytope. 
If the moment polytope is a flagtope then the corresponding small cover is aspherical (see \cite[Theorem 2.2.5]{MTR}). 

Hausmann and Roudrigue \cite[Proposition 6.8]{JE} establish a sufficient condition for spatial polygon space  $\mathrm{N}_{\alpha}$ to be a toric manifold, which we now state. 

\begin{theorem}\label{scrtv}
Let $\alpha=(\alpha_{1},\alpha_{2},\dots,\alpha_{m})$ be a generic length vector with $\alpha_{m}\geq  \sum_{i=1}^{m-5}\alpha_{i}$ then $\mathrm{N}_{\alpha}$ is diffeomorphic to a toric manifold.
\end{theorem}

Recall that $\overline{\mathrm{M}}_{\alpha}$ is a fixed point set of an anti-syplectic involution on $\mathrm{N}_{\alpha}$. 
Hence when $\alpha$ satisfies the condition of \Cref{scrtv} corresponding polygon space is a small cover.

\begin{remark}\label{tvobs}
{\normalfont For a generic length vector $\alpha=(\alpha_{1},\alpha_{2},\dots,\alpha_{m})$, the half-dimensional torus action on $\mathrm{N}_{\alpha}$ is given by bending flows (see \cite[Section 5]{hkgrasmann}, \cite[Section 6]{JE}).
They described the moment polytope using triangle inequalities.

However, in general it is hard to characterize the face poset of the moment polytope using those inequalities. 
For example, we cannot determine the number of facets of the moment polytope, one of the important information to determine whether a simple polytope is flagtope or not.
So, whenever $\mathrm{N}_{\alpha}$ is a toric manifold, it is not straightforward to use \Cref{artv} to conclude whether $\overline{\mathrm{M}}_{\alpha}$ is aspherical or not. }
\end{remark}

In view of this discussion, we now focus on a particular sub-class of planar polygon spaces, called chain spaces (or abelian polygon spaces), because they admit a torus action and it is possible to explicitly describe their moment polytope.

\section{The short code of a chain space}\label{tsc}

As stated above we now focus solely on chain spaces (i.e., real abelian polygon spaces). 
In this section, we introduce the notion of a combinatorial object associated with generic length vector, called the short code. 
The short code of a generic length vector is closely related to the \emph{genetic code} defined by Hausmann in \cite[Section 1.5]{hgdps} in the context of polygon spaces.
We also show that chain spaces are planar polygon spaces, up to diffeomorphism.

Since the diffeomorphism type of a chain space does not depend on the ordering of its side lengths, we assume that our (generic) length vector $\alpha=(\alpha_{1},\alpha_{2},\dots,\alpha_{m-1},\alpha_{m})$ satisfies $\alpha_{1}\leq \alpha_{2} \leq \dots\leq \alpha_{m-1}$. 
Note that the only restriction on $\alpha_m$ is that it is less than the sum $\sum_{i=1}^{m-1} \alpha_i$.

\begin{definition}\label{tscmglv}
For a generic length vector $\alpha$, we define the following collection of subsets of $[m]$:
$$ S_{m}(\alpha) = \{J\subset [m] : m\in J \hspace{1mm} \text{and}\hspace{1mm} J \hspace{1mm} \text{is}\hspace{1mm} \alpha-\text{short}\}. $$

A partial order $\leq$ is defined on $S_{m}(\alpha)$ by declaring $I\leq J$ if $I=\{i_{1},\dots,i_{t}\}$ and $\{j_{i},\dots,j_{t}\}\subseteq J$ with $i_{s}\leq j_{s}$ for $1\leq s\leq t$. 
The \emph{short code} of $\alpha$ is the set of maximal elements of $S_{m}(\alpha)$ with respect to this partial order. 
If $A_{1}, A_{2},\dots, A_{k}$ are the maximal elements of $S_{m}(\alpha)$, we denote the short code as $<A_{1},\dots,A_{k}>$.
\end{definition}

\begin{example}
The short codes of length vectors $(1,1,3,3,3)$ and $(1,2,2,5,3)$ are $<\{1,2,5\}>$ and $<\{1,3,5\}>$ respectively.
\end{example}

Given a generic length vector $\alpha=(\alpha_{1},\alpha_{2},\dots,\alpha_{m-1},\alpha_{m})$ of a chain space and a positive real number $\delta > \sum_{i=1}^{m-1}\alpha_{i}$, define a new generic length vector (for a polygon space) as

\begin{equation}\label{alphaprime}
\alpha^{\prime} = (\alpha_{1},\alpha_{2},\dots,\alpha_{m-1},\delta,\alpha_{m}+\delta).
\end{equation}

It was shown in \cite{JA} that the spatial polygon space $\mathrm{N}_{\alpha^{\prime}}$ is a toric manifold. 
In particular, the authors proved that the moment polytope of the abelian polygon space corresponding to $\alpha$ and that of $\mathrm{N}_{\alpha^{\prime}}$ are isomorphic \cite[Proposition 1.3]{JA}. 
Hence the two spaces are diffeomorphic. 
Here we prove the planar version of their result by providing an explicit diffeomorphism.

\begin{proposition}\label{cspps}
Let $\alpha$ be a generic length vector and $\alpha^{\prime}$ be the vector defined in \Cref{alphaprime}. 
Then the corresponding chain space $\mathrm{Ch}(\alpha)$ is diffeomorphic to the planar polygon space $\overline{\mathrm{M}}_{\alpha^{\prime}}$. 
\end{proposition}
\begin{proof}
 Given a chain with side lengths $\alpha_{1},\alpha_{2},\dots,\alpha_{m-1}$ we can associate unique polygon with side lengths $\alpha_{1},\alpha_{2}\dots\alpha_{m-1},\delta,\alpha_{m}+\delta$ in the following way: 

\begin{figure}[H]\label{fig1}
	\centering
	$\includegraphics[scale=1 ]{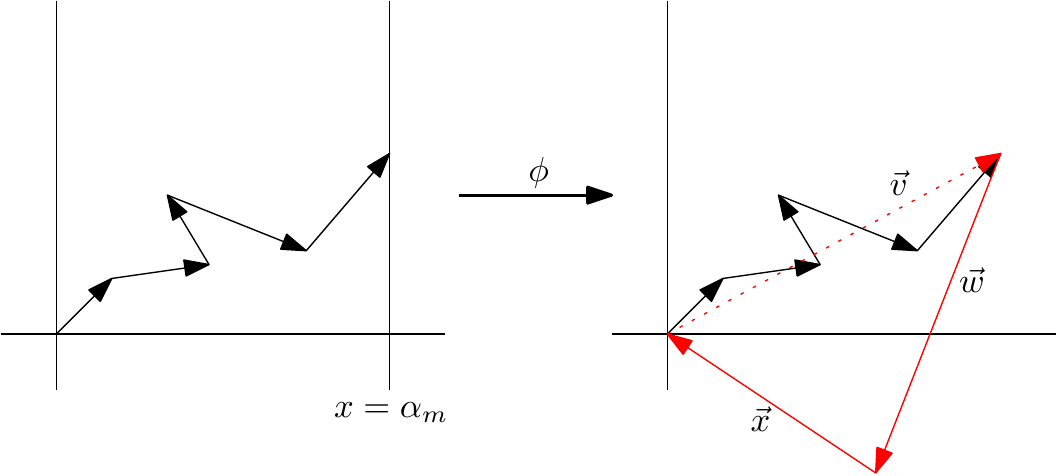}$
	\caption{}
\end{figure}

Let $(v_{1},v_{2},\dots, v_{m-1})\in \mathrm{Ch}(\alpha)$ and let $\vec{v}=\sum_{i=1}^{m-1} v_{i}$. 
We have following inequalities : $$|\vec{v}| + (\delta + \alpha_{m}) > \delta , \hspace{.1cm} \delta + (\delta + \alpha_{m}) > |\vec{v}|, \hspace{.1cm} |\vec{v}| + \delta > \delta + \alpha_{m}$$ as $\vec{v}=(\alpha_{m},y)$ for some $y\in \mathbb{R}$. 
Therefore, the sides lengths $|\vec{v}|$, $\delta$ and $\delta + \alpha_{m}$ satisfy the triangle inequalities. 
In fact there exist unique triangle (see \Cref{fig1}) with direction vectors $\vec{v}$, $\vec{w}$ and $\vec{x}$ up to isometries. 
Consequently, we have an $m+1$-gon $(v_{1},v_{2},\dots,v_{m-1},\vec{w},\vec{x})$ with side length $\alpha_{1},\alpha_{2},\dots,\alpha_{m-1},\delta,$  $\alpha_{m}+\delta$.  
Its not hard to see that the map $$\phi : \mathrm{Ch}(\alpha)\longrightarrow \overline{\mathrm{M}}_{\alpha^{\prime}} $$ defined by $$\phi((v_{1},v_{2},\dots,v_{m-1}))=(v_{1},v_{2},\dots,v_{m-1},\vec{w},\vec{x})$$ is a diffeomorphism.   
\end{proof} 

\begin{remark}\label{scgc}
If the short code of a generic length vector $\alpha=(\alpha_{1},\alpha_{2},\dots,\alpha_{m-1},\alpha_{m})$ is $<A_{1}\cup\{m\},A_{2}\cup\{m\},\dots,A_{k}\cup\{m\}>$ where $A_{i}\subseteq [m-1]$ for all $1\leq i\leq k$ then $<A_{1}\cup\{m+1\},A_{2}\cup\{m+1\},\dots,A_{k}\cup\{m+1\}>$ is the genetic code of the length vector $\alpha^{\prime}$ defined in Equation \ref{alphaprime}.
\end{remark}

\begin{proposition}
Let $\alpha$ and $\beta$ be two generic length vectors with the same short code. Then the corresponding chain spaces $\mathrm{Ch}(\alpha)$ and $\mathrm{Ch}(\beta)$ are diffeomorphic.
\end{proposition}
\begin{proof}
Consider $\alpha^{\prime}$ and $\beta^{\prime}$ be two generic length vectors defined as in \Cref{alphaprime}. 
Note that the genetic codes of $\alpha^{\prime}$ and $\beta^{\prime}$ coincide, since the short code of $\alpha$ and $\beta$ are same. 
Consequently \cite[Lemma 4.2]{JE} and \cite[Lemma 1.2 ]{hgdps} give that the the planar polygon spaces $\overline{\mathrm{M}}_{\alpha^{\prime}}$ and $\overline{\mathrm{M}}_{\beta^{\prime}}$ corresponding to $\alpha^{\prime}$ and $\beta^{\prime}$ are diffeomorphic. 
It follows from the \Cref{cspps} that $\mathrm{Ch}(\alpha^{\prime})$ and $\mathrm{Ch}(\beta^{\prime})$ are diffeomorphic.
\end{proof}

\begin{corollary}\label{m-2torus}
Let $<\{1,2,\dots,m-2,m\}>$ be the short code of a length vector $\alpha$. Then the corresponding chain space $\mathrm{Ch}(\alpha)$ is diffeomorphic to the $(m-2)$-dimensional torus $T^{m-2}$.  
\end{corollary}
\begin{proof}
Let $\alpha^{\prime}$ be the new length vector defined as in Equation \ref{alphaprime}. 
Using \Cref{scgc} we get the genetic code of $\alpha^{\prime}$; which is $<\{1,2,\dots,m-2,m+1\}>$.
By \cite[Theorem 1]{MJ} we infer that $\overline{\mathrm{M}}_{\alpha^{\prime}}$ is diffeomorphic to $T^{m-2}$. 
Now the corollary follows from Proposition \ref{cspps}.
\end{proof}


\section{The face poset of the moment polytope}\label{fpmp}

Recall that the chain space $\mathrm{Ch}(\alpha)$ is the real part of a toric manifold (see the paragraph after \Cref{dcs}). 
Our aim is to understand the topology of this space via the combinatorics of the associated moment polytpoe. 
In this section we describe the face poset of the moment polytope. 
We show that it is completely  determined by the short code of the corresponding length vector. 
The moment polytopes corresponding to two different length vectors are strongly isomorphic if they have the same short code. 
In particular, the results in this section imply that if the short codes are same then the equations defining the corresponding moment polytopes differ by a constant (equivalently their facets are parallel). 

Given a generic length vector $\alpha$, the moment polytope of the corresponding chain space was first described in \cite{JA} as an intersection of a parallelepiped with a hyperplane. 
We begin with the description of the moment map:
\[\mu: \mathrm{Ch}(\alpha)\longrightarrow \mathbb{R}^{m-1} \]
such that 
$$\mu(v_{1},\dots,v_{m-1})=(\alpha_{1}x_{1},\dots,\alpha_{m-1}x_{m-1})$$ 
where $v_{i}=(x_{i},y_{i})$. 
We have $$\mathrm{P}(\alpha) := \mu(\mathrm{Ch}(\alpha))=\{(w_{1},w_{2},\dots,w_{m-1})\in \prod_{i=1}^{m-1} [-\alpha_{i},\alpha_{i}] : \sum_{i=1}^{m-1} w_{i}=\alpha_{m}\}.$$ 
Let $$\mathrm{C}^{m-1}(\alpha)=\displaystyle\prod_{i=1}^{m-1} [-\alpha_{i},\alpha_{i}]$$ and  $$\mathrm{H}(\alpha)=\{(w_{1},w_{2},\dots,w_{m-1})\in \mathbb{R}^{m-1}:\displaystyle \displaystyle\sum_{i=1}^{m-1}w_{i}=\alpha_{m}\}.$$
Hence the moment polytope $\mathrm{P}(\alpha)=\mathrm{C}^{m-1}(\alpha)\cap \mathrm{H}(\alpha)$. 

It is clear that the facets of $\mathrm{P}(\alpha)$ are given by intersections of the facets of $\mathrm{C}^{m-1}(\alpha)$ with the hyperplane $\mathrm{H}(\alpha)$. 
Note that the facets of $\mathrm{C}^{m-1}(\alpha)$ are described by the equations $x_{j}=\pm \alpha_{j}$. We call the facets described by equations $x_{j}=\alpha_{j}$ and $x_{j}= -\alpha_{j}$ as \emph{positive facets} and \emph{negative facets} respectively.

Corollary \ref{m-2torus} says that, if the short code of a length vector is $<\{1,2,\dots,m-2,m\}>$ then the corresponding chain space is diffeomorphic to $(m-2)$-dimensional torus. 
So it is clear that the corresponding moment polytope  must be a $(m-2)$-cube. 
Here we explicitly describe it in terms of intersections of half-spaces. 

\begin{lemma}
Let $<\{1,2,\dots,m-2,m\}>$ be the short code of a length vector $\alpha$. 
Then the moment polytope $\mathrm{P}(\alpha)\cong I^{m-2}$, the $(m-2)$-dimensional cube.
\end{lemma}

\begin{proof}
Let $F_{m-1}$ and $\bar{F}_{m-1}$ be the two opposite facets of $\mathrm{C}^{m-1}(\alpha)$ represented by the equations $x_{m-1}=\alpha_{m-1}$ and $x_{m-1}=-\alpha_{m-1}$ respectively. 
Note that the set 
$$\{(\pm \alpha_{1},\pm \alpha_{2},\dots,\pm \alpha_{m-2},\alpha_{m-1})\}$$ 
forms the vertices of $F_{m-1}$ and the set
$$\{(\pm \alpha_{1},\pm \alpha_{2},\dots,\pm \alpha_{m-2},-\alpha_{m-1})\}$$ 
forms the vertices of $\bar{F}_{m-1}$. 
Since $\{m-1\}$ is a long subset we have the following inequality, 
\begin{equation}\label{eosl}
  -\alpha_{m-1}+\sum_{i=1}^{m-2}\pm \alpha_{i}< \alpha_{m}< \alpha_{m-1}+\sum_{i=1}^{m-2}\pm \alpha_{i}.
\end{equation}

The rightmost inequality above implies that, the hyperplane $\mathrm{H}(\alpha)$ does not intersect the facet $\bar{F}_{m-1}$. 
Similarly the leftmost inequality says that $\mathrm{H}(\alpha)$ does not intersect the facet $F_{m-1}$. 
Therefore, the hyperplane $\mathrm{H}(\alpha)$ passes through $\mathrm{C}^{m-1}(\alpha)$ dividing it into two isomorphic polytopes which are combinatorially isomorphic to $(m-1)$-cubes. 
We conclude that $C^{m-1}(\alpha)\cap \mathrm{H}(\alpha) \cong I^{m-2}$, being a facet of both the divided parts.
\end{proof}

Since our aim is to classify the aspherical chain spaces, henceforth we discard length vectors whose short code is $<\{1,2,\dots,m-2,m\}>$, i.e., we discard the length vector where $\alpha_{m-1}>\sum_{i=1}^{m-2}\alpha_{i} + \alpha_{m}$. 
In particular we assume that  $\alpha_{i}<\sum_{j\neq i}\alpha_{j}$.

\begin{proposition} \label{m-1facet}
Let $\alpha=(\alpha_{1},\dots,\alpha_{m})$ be a length vector with $\alpha_{i}<\sum_{j\neq i}\alpha_{j}$. Then the hyperplane $\mathrm{H}(\alpha)$ intersects all the positive facet of $\mathrm{C}^{m-1}(\alpha)$.
\end{proposition}

\begin{proof}

We argue by contradiction. 
Suppose that the hyperplane $\mathrm{H}(\alpha)$ does not intersect all the positive facets.
In particular, let it miss the facet given by the equation $x_{j_{0}}=\alpha_{j_{0}}$. 
The coordinate sum of the elements of $\mathrm{H}(\alpha)$ (which is $\alpha_{m}$) is then strictly less than 
the coordinate sum of points on the missed facet, i.e.,  $$\alpha_{m} < \alpha_{j_{0}}-\sum_{i\neq j_{o}}\alpha_{i}$$ 
equivalently
$$\alpha_{m}+\sum_{i\neq j_{o}}\alpha_{i} < \alpha_{j_{0}}.$$ 
This is impossible since $\alpha$  is generic.  
\end{proof}

The \Cref{m-1facet} describes $m-1$ many facets of the moment polytope. 
We now characterize length vectors $\alpha$ when $\mathrm{P}(\alpha)$ is a simplex.

\begin{lemma}\label{jml}
Let $\alpha=(\alpha_{1},\alpha_{2},\dots,\alpha_{m})$ be a generic length vector.  The hyperplane	$\mathrm{H}(\alpha)$ intersects exactly $m-1$ facets of $\mathrm{C}^{m-1}(\alpha)$ given by an equations $x_{j}=\alpha_{j}$ for $1\leq j \leq m-1$
if and only if $\{m,j\}$ is long subset for all $1\leq j\leq m-1$. In particular $\mathrm{P}(\alpha)\cong \bigtriangleup^{m-2}$, an $(m-2)$-simplex.
\end{lemma}

\begin{proof}
Consider a vertex $v=(\alpha_{1},\dots,\alpha_{m-1})$ of $\mathrm{C}^{m-1}(\alpha)$ and $$\{v(j)=(\alpha_{1},\dots,-\alpha_{j},\dots,\alpha_{m-1}) : 1\leq j \leq m-1\}$$ be its neighboring vertices. We can observe that the hyperplane $\mathrm{H}(\alpha)$ cannot intersects $v$ since $\alpha$ is generic.
Note that $\mathrm{H}(\alpha)$ intersects exactly $m-1$ positive facets of $\mathrm{C}^{m-1}(\alpha)$ if and only if the coordinate sum (which is $\alpha_{m}$) of an element of $\mathrm{H}(\alpha)$ is 
greater than or equal the coordinate sum of adjacent vertices of $v(j)$.  
Let $\beta_{j}=\sum_{i\neq j}\alpha_{i}-\alpha_{j}$ be the coordinate sum of $v(j)$ for $1\leq j\leq m-1$. The hyperplane $\mathrm{H}(\alpha)$ intersects exactly $m-1$ positive facets of $\mathrm{C}^{m-1}(\alpha)$ if and only if
$$\alpha_{m}\geq \beta_{j}=\sum_{i\neq j}\alpha_{i}-\alpha_{j},$$ for all $1\leq j\leq m-1$. 
Which gives $$\alpha_{m}+\alpha_{j}\geq \sum_{i\neq j}\alpha_{i}.$$  This proves the lemma.
\end{proof}

The next result determines the remaining facets of the moment polytope. 
It shows that these facets can be determined using the short subset information.

\begin{lemma}\label{jms}
Let $\alpha=(\alpha_{1},\alpha_{2},\dots,\alpha_{m})$ be a generic length vector.
The hyperplane $\mathrm{H}(\alpha)$ intersects a facet of $\mathrm{C}^{m-1}(\alpha)$ given by an equation $x_{j}=-\alpha_{j}$ for some $1\leq j \leq m-1$ if and only if $\{m,j\}$ is an $\alpha$-short subset.
\end{lemma}

\begin{proof}
Let $v(j)=(\alpha_{1},\dots,-\alpha_{j},\dots,\alpha_{m-1})$ and $w=(-\alpha_{1},\dots, +\alpha_j,\dots,-\alpha_{m-1})$ be the two (extreme) vertices of the facet supported by $x_{j}=-\alpha_{j}$. 
The coordinate sum of $v(j)$ is 
$$\sum_{i\neq j}\alpha_{i}-\alpha_{j}.$$
The hyperplane $\mathrm{H}(\alpha)$ intersects the facet supported by $x_{j}=-\alpha_{j}$ if and only if the sum of the coordinates of an element of $\mathrm{H}(\alpha)$ is
between the  coordinate sums of points $v(j)$ and $w$, i.e., 
$-\sum_{i}^{m-1}\alpha_{i} \leq \alpha_{m} \leq \sum_{i\neq j}\alpha_{i}-\alpha_{j}$, equivalently   
$$\alpha_{m}+\alpha_{j} \leq \sum_{i\neq j}\alpha_{i}.$$ This proves the lemma.
\end{proof}

Both the \Cref{jml} and \Cref{jms} give the exact count of the facets in terms of the short subset information. 

\begin{lemma}\label{FACET}
Let $\alpha=(\alpha_{1},\alpha_{2},\dots,\alpha_{m})$ be a generic length vector and $k$ be the number of two element $\alpha$-short subsets containing $m$. Then $\mathrm{P}(\alpha)$ has $m-1+k$ many facets.
\end{lemma}{}
\begin{proof}
Follows from the \Cref{jml} and \Cref{jms}. 
\end{proof}{}

We now describe the face poset of the moment polytope $\mathrm{P}(\alpha)$ in terms of certain subsets of $[m]$. 
We further show that this face poset is completely determined by the short code. 

Define $[m-1]^{+}\coloneqq \{1,2,\dots,m-1\}$ and $[m-1]^{-}\coloneqq \{\Bar{1},\Bar{2},\dots,\overline{m-1}\}$ where $\Bar{i}=-i$. 
Denote by $[m-1]^{\pm}$ the union $[m-1]^{+}\cup [m-1]^{-}$. 
Let $J_{1}\subseteq [m-1]^{+}$ and $J_{2}\subseteq [m-1]^{-}$. 
We write $\alpha_{J_{1}}$ for $\sum_{j_{i}\in J_{1}}\alpha_{j_{i}}$ and $\alpha_{J_{2}}$ for $\sum_{j_{s}\in -J_{2}}\alpha_{j_{s}}$.

\begin{definition}\label{admsub}
Let $J\subseteq [m-1]^{\pm}$, such that $J=J_{1}\cup J_{2}$ where $J_{1}\subseteq [m-1]^{+}$, $J_{2}\subseteq [m-1]^{-}$ and $J_{1}\cap -J_{2}=\emptyset$. 
Then $J$ is called an \emph{admissible subset} of $[m-1]^{\pm}$ if the shorter length vector 
$$\alpha(J) := (\alpha_{j_{1}},\alpha_{j_{2}},\dots,\alpha_{j_{k}},|\alpha_{m}+ \alpha_{J_{1}}-\alpha_{J_{2}}|)$$ 
is generic, where $(J_{1}\cup -J_{2})^{c}=\{j_{1},\dots,j_{k}\}$.
\end{definition}

We denote the set of all admissible subsets of $[m-1]^{\pm}$ by $\mathrm{Ad}(\alpha)$.
We have a natural partial order on $\mathrm{Ad}(\alpha)$,  $J\leq S$  if and only if $S\subseteq J$. 
With this partial order $\mathrm{Ad}(\alpha)$ becomes a poset.

\begin{theorem}\label{POSET}
The poset $\mathrm{Ad}(\alpha)$ is isomorphic to the face poset of  $\mathrm{P}(\alpha)$.
\end{theorem}
\begin{proof}
We start by associating a face of the moment polytope to an admissible subset of $[m-1]^{\pm}$. 
Let $J$ be an admissible subset of $[m-1]^{\pm}$. Define $$F_{J}\coloneqq \{(x_{1},x_{2},\dots,x_{m-1})\in \mathrm{P}(\alpha): x_{j_{i}}=\alpha_{j_{i}},j_{i}\in J_{1} , x_{j_{s}}=-\alpha_{j_{s}},j_{s}\in -J_{2} \}$$ Let $F_{i}$ and $\Bar{F}_{i}$ be the facets of $\mathrm{C}^{m-1}(\alpha)$ given by an equations $x_{i}=\alpha_{i}$ and $x_{i}=-\alpha_{i}$. Note that $$F_{J}=\displaystyle\bigcap_{j_{i}\in J_{1}}F_{j_{i}}\bigcap \displaystyle\bigcap_{j_{s}\in -J_{2}}\Bar{F}_{j_{s}}\bigcap \mathrm{H}(\alpha).$$ 
In fact it is easy to see that any face of $\mathrm{P}(\alpha)$ looks like $F_{J}$ for some admissible subset of $[m-1]^{\pm}$. 
The map $$\phi: \mathrm{Ad}(\alpha)\longrightarrow \mathrm{P}(\alpha)$$ defined by $$\phi(J)=F_{J},$$  is a poset isomorphism.
\end{proof}

\begin{remark}\label{addsc}
We observe straightforwardly that
for an admissible subset $J$,  $\mathrm{dim}(F_{J})=m-2-|J|$.
Moreover, let $\alpha$ and $\beta$ be two generic length vectors with an isomorphism between $\mathrm{Ad}(\alpha)$ and $\mathrm{Ad}(\beta)$. 
Any such isomorphism restricts to an isomorphism between $S_{m}(\alpha)$ and $S_{m}(\beta)$. 
Therefore, the short code of $\alpha$ determines the poset $\mathrm{Ad}(\alpha)$.
\end{remark}

\begin{proposition}
Let $<\{k-1,m\}>$ and $<\{k,m\}>$ be short codes of generic length vectors $\alpha$ and $\beta$ respectively. Then the moment polytope $\mathrm{P}(\beta)$ is obtained by truncating a vertex of $\mathrm{P}(\alpha)$. 
\end{proposition}
\begin{proof}
Let $Z=\{\overline{1},\overline{2},\dots,\overline{k-1},\overline{k+1},\dots,\overline{m-1}\}$ be an admissible subset corresponding to $\alpha$. Note that $Z$ represents a vertex of $\mathrm{P}(\alpha)$. Let $Z_{i}=Z\setminus \{\overline{i}\}\cup \{k\}$ for all $i\in \{1,2,\dots,k-1,k+1,\dots,m-1\}$. Each $Z_{i}$ is an admissible subset corresponding to $\beta$ and represents a vertex of $\mathrm{P}(\beta)$.
Let $\mathrm{Ad}(\alpha)$ and $\mathrm{Ad}(\beta)$ be the posets of an admissible subsets corresponding to the length vector $\alpha$, $\beta$ respectively. Then we have,  
\begin{equation}\label{adbeta}
    \mathrm{Ad}(\beta) = \mathrm{Ad}(\alpha)\setminus Z \bigsqcup \{\{k\}, J : J\subseteq Z_{i}\},
\end{equation}
where $i\in \{1,2,\dots,k-1,k+1,\dots,m-1\}$. Note that $\{\{k\}, J : J\subseteq Z_{i}\}$ represents an $(m-3)$-simplex in $\mathrm{P}(\beta)$. 
The \Cref{adbeta} implies the vertex $Z$ of $\mathrm{P}(\alpha)$ is replaced by an $(m-3)$-simplex represented by the admissible subset $\{k\}$ thus proving the proposition. 
\end{proof}

\begin{example} 
If the short code of a generic length vector is $<\{5\}>$ then the corresponding moment polytope is a $3$-simplex.
On the other hand, if the short code is $<\{1,5\}>$ then the corresponding moment polytope is obtained by truncating a vertex of the $3$-simplex.  
\begin{figure}[H]
\centering
\includegraphics[scale=1]{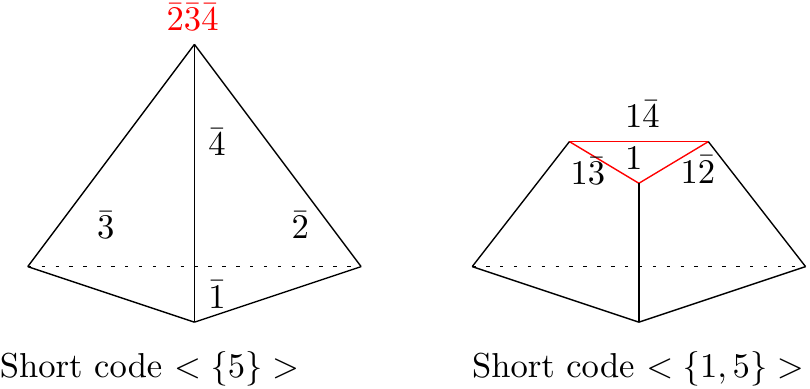}
\caption{Vertex truncation}
\end{figure}
\end{example}

Now we show that the poset of admissible subsets $\mathrm{Ad}(\alpha)$ is determined by the short code of the corresponding length vector.  

\begin{theorem}\label{Add}
Let $\alpha=(\alpha_{1},\dots,\alpha_{m})$ be a generic length vector. 
Let $J=J_{1}\cup J_{2}$ where $J_{1}\subseteq [m-1]^{+}$ and $J_{2}\subseteq [m-1]^{-}$ with $J_{1}\cap J_{2}= \emptyset$. 
Then $J$ is an admissible subset of $[m-1]^{\pm}$ if and only if $J_{1}\cup \{m\}$ and $-J_{2}$ are $\alpha$-short subsets of $[m]$. 
\end{theorem}

\begin{proof}
Suppose $J=J_{1}\cup J_{2}$ is an admissible subset of $[m-1]^{\pm}$. i.e. 
\begin{equation}\label{adsb}
|\alpha_{m}+\sum_{i\in J_{1}}\alpha_{i}-\sum_{j\in -J_{2}}\alpha_{j}|< \sum_{k\in (J_{1}\cup -J_{2})^c}\alpha_{k}.	    
\end{equation}
The above equation gives us two inequalities, the first of which is 
\begin{equation}\label{in1}
    \alpha_{m}+\sum_{i\in J_{1}}\alpha_{i}-\sum_{j\in J_{2}}\alpha_{j}< \sum_{k\in (J_{1}\cup -J_{2})^c}\alpha_{k}
\end{equation}
and the second one is
\begin{equation}\label{in2}
 -\sum_{k\in (J_{1}\cup -J_{2})^c}\alpha_{k}< \alpha_{m}+\sum_{i\in J_{1}}\alpha_{i}-\sum_{j\in J_{2}}\alpha_{j}.
\end{equation}
\Cref{in1} gives 
$$\alpha_{m}+\sum_{i\in J_{1}}\alpha_{i}< \sum_{k\in (J_{1}\cup -J_{2})^c}\alpha_{k}+ \sum_{j\in J_{2}}\alpha_{j},$$ 
i.e., $J_{1}\cup \{m\}$ is a short subset. 
\Cref{in2} gives 
$$\sum_{j\in J_{2}}\alpha_{j}<\alpha_{m}+\sum_{i\in J_{1}}\alpha_{i} + \sum_{k\in (J_{1}\cup -J_{2})^c}\alpha_{k},$$ 
i.e., $-J_{2}$ is a short subset. 

Conversely we can obtain the \Cref{adsb} using the inequalities in \Cref{in1} and \Cref{in2}. 
It proves the converse since the length vector $\alpha(J)$ is generic.
As the short code determines the collection of short subsets, it automatically determines the poset of admissible subsets. 
\end{proof}

Recall the definition of \emph{small covers} from the Introduction. 
Note that these manifolds are a topological generalization of real toric varieties as proved by Davis and Januszkiewicz in \cite{MT}. 
One of their important result specifies how to build a small cover from the quotient polytope (see \cite[Section 1.5]{MT} for details). 
It says that there is a regular cell structure on the manifold consisting of $2^n$ copies of the quotient polytope as the top-dimensional cells. 
We describe their construction briefly.

Consider an $n$-dimensional simple polytope $P$ with the facet set $\mathcal{F} = \{F_1,\dots, F_m\}$. 
A function $\chi:\mathcal{F}\to \mathbb{Z}_2^n$ is called characteristic for $P$ if for each vertex $v= F_{i_1}\cap\cdots\cap F_{i_n}$, the $n\times n$ matrix whose columns are $\chi(F_{i_1}),\dots,\chi(F_{i_n})$ is unimodular. 
Given the pair $(P, \chi)$ the corresponding small cover $X(P, \chi)$ is constructed as follows:
\[X(P, \chi) := \left((\mathbb{Z}_2)^n\times P\right)/ \{(t,p) \sim (u,q) \}\quad \text{if~} p = q \text{~and~} t^{-1}u \in \mathrm{stab}(F_q)\]
where $F_q$ is the unique face of $P$ containing $q$ in its relative interior. 

In the remaining section we explicitly determine the entries of the characteristic matrix for chain spaces. 
The moment polytope $\mathrm{P}(\alpha)$ is $(m-2)$-dimensional but described as a subset of $\mathbb{R}^{m-1}$, so project it onto an affinely isomorphic polytope $\mathrm{Q}(\alpha)$. 
This new polytope is embeded in $\mathbb{R}^{m-2}$ and we determine its outward normals, which determine the characteristic function. 

Given a generic length vector $\alpha=(\alpha_{1},\dots,\alpha_{m-1},\alpha_{m})$, we define the hyperplanes in $\mathbb{R}^{m-2}$ as follows: 
$$\Bar{H}_{m-1}(\alpha)=\{(x_{1},x_{2},...,x_{m-2})\in \mathbb{R}^{m-2}:\sum_{i=1}^{m-2}x_{i}=\alpha_{m}-\alpha_{m-1}\}$$ and $$H_{m-1}(\alpha)=\{(x_{1},x_{2},...,x_{m-2})\in \mathbb{R}^{m-2}:\sum_{i=1}^{m-2}x_{i}=\alpha_{m}+\alpha_{m-1}\}.$$ Let $\Bar{H}_{m-1}^{\geq 0}$ and $H_{m-1}^{\leq 0}$ be a positive and negative part of $\Bar{H}_{m-1}$ and $H_{m-1}$ respectively.

\begin{theorem}
Define an $(m-2)$-dimensional polytope as follows
$$
\mathrm{Q}(\alpha)	=\begin{cases}
	\displaystyle\prod_{i=1}^{m-2}[-\alpha_{i},\alpha_{i}]\cap \Bar{H}_{m-1}^{\geq 0}, & \text{if $\{m-1,m\}$ is long subset,}\\
	\displaystyle\prod_{i=1}^{m-2}[-\alpha_{i},\alpha_{i}]\cap \Bar{H}_{m-1}^{\geq 0}\cap H_{m-1}^{\leq o}, & \text{if $\{m-1,m\}$ is short subset.}
	\end{cases}
$$

Then $\mathrm{P}(\alpha)$ is affinely isomorphic to  $\mathrm{Q}(\alpha)$.
\end{theorem}	

\begin{proof}
	Let  $$\pi: \mathbb{R}^{m-1}\longrightarrow \mathbb{R}^{m-2}$$ be the projection defined by $$\pi(x_{1},x_{2},\dots,x_{m-1})=(x_{1},x_{2},\dots,x_{m-2}).$$ 
	Note that $\pi$ restricts to an isomorphism on $\mathrm{H}(\alpha)$. 
	Since $\mathrm{P}(\alpha)\subseteq \mathrm{H}(\alpha)$, $\pi$ gives a bijection between $\mathrm{P}(\alpha)$ and $\pi(\mathrm{P}(\alpha))$. 
	We show that $\mathrm{Q}(\alpha)=\pi(\mathrm{P}(\alpha))$. 
	Suppose  $\{m-1,m\}$ is a long subset. 
	Then $\mathrm{H}(\alpha)$ does not intersects a facet of $\mathrm{C}^{m-1}(\alpha)$ given by $x_{m-1}=-\alpha_{m-1}$. 
	Also $H_{m-1}$ does not intersect $\mathrm{C}^{m-2}(\alpha)$. 
	Let $(x_{1},x_{2},\dots,x_{m-1})\in \mathrm{P}(\alpha)$. 
	Then the following inequality is clear: 
	$$\sum_{i=1}^{m-2}x_{i}=\alpha_{m}-x_{m-1}\geq \alpha_{m}-\alpha_{m-1}.$$ 
	This gives us $\pi(\mathrm{P}(\alpha))\subseteq \mathrm{Q}(\alpha)$. 
	Now we show the other inclusion. Let $(y_{1},y_{2},\dots,y_{m-2})\in \mathrm{Q}(\alpha)$. 
	Then it follows that $\displaystyle \sum_{i=1}^{m-2}y_{i}\geq \alpha_{m}-\alpha_{m-1}$. Let $a = \alpha_{m}-\displaystyle\sum_{i=1}^{m-2}y_{i}$ and $\Bar{y}=(y_{1},y_{2},\dots,y_{m-2},a)$. 
	Note that $\Bar{y}\in \mathrm{P}(\alpha)$ as $\displaystyle\sum_{i=1}^{m-2}y_{i}+a = \alpha_{m}$ and $|a|\leq \alpha_{m-1}$. 
	Since $\pi(\Bar{y})=(y_{1},y_{2},\dots,y_{m-2})$, $(y_{1},y_{2},\dots,y_{m-2})\in \pi(\mathrm{P}(\alpha))$. 
	Which gives $\mathrm{Q}(\alpha) \subseteq \pi(\mathrm{P}(\alpha))$. We conclude that $\mathrm{Q}(\alpha)=\pi(\mathrm{P}(\alpha))$.
	Similar arguments works when $\{m-1,m\}$ is a short subset.
\end{proof}

Note that the facets of $\mathrm{Q}(\alpha)$ are given by following equations.

\begin{itemize}
    \item $\mathrm{F}_{i} : x_{i}=\alpha_{1}$ for $1\leq i \leq m-2$.
    \item $\overline{\mathrm{F}}_{m-1} : \displaystyle \sum_{i=1}^{m-2}x_{i}= \alpha_{m}-\alpha_{m-1}$ when $\{m-1,m\}$ is long subset.
    \item $\mathrm{F}_{i} : x_{i}=-\alpha_{i}$ when $\{i,m\}$ is short subset.
    \item $\mathrm{F}_{m-1} : \displaystyle\sum_{i=1}^{m-2}x_{i}=\alpha_{m}+\alpha_{m-1}$ when $\{m-1,m\}$ is short subset.
\end{itemize}

Let $\mathscr{F}(\mathrm{Q}(\alpha))$ be the collection of facets of $\mathrm{Q}(\alpha)$. We define a map $$\chi_{\alpha} : \mathscr{F}(\mathrm{Q}(\alpha)) \longrightarrow \mathbb{Z}_{2}^{m-2}$$ as follows: 
\[ \chi_{\alpha}(\overline{\mathrm{F}}_{i})=
    \begin{dcases}
        -e_{i}, & 1\leq i \leq m-2, \\
        \displaystyle\sum_{i=1}^{m-2}e_{i}, & \text{~if~}\{m-1,m\} \hbox{~is long~},
    \end{dcases}
\]
and 
\[ \chi_{\alpha}(\mathrm{F}_{i})=
    \begin{dcases}
        e_{i}, & \hbox{~if~} \{i,m\} \hbox{~is short~},\\
        -\displaystyle\sum_{i=1}^{m-2}e_{i}, & \hbox{~if~} \{m-1,m\} \hbox{~is short~}.
    \end{dcases}
\]

The following result is clear. 
\begin{lemma}
The function $\chi_{\alpha}$ is characteristic for $\mathrm{Q}(\alpha)$.
\end{lemma}


We can now state the main equivalence. 
\begin{theorem}
With the notation as above the chain space $\mathrm{Ch}(\alpha)$ is the small cover corresponding to the pair $(\mathrm{Q}(\alpha), \chi_{\alpha})$. 
\end{theorem}
\begin{proof}
The natural action of $\mathbb{Z}_{2}^{m-1}$ on $(S^{1})^{m-1}$ descends to $\mathrm{Ch}(\alpha)$.
However, this action is not effective. 
The element $(r,\dots,r)$ of $\mathbb{Z}_{2}^{m-1}$ (where $r$ generates a copy of $\mathbb{Z}_2$) fixes every element of $\mathrm{Ch}(\alpha)$.
So we get the effective action by dividing by the diagonal subgroup. 
The remaining details are easy to verify hence omitted here. 
\end{proof}

\begin{remark}
Note that the above charcterisitc vectors are normal to the corresponding facets. 
Hence, if all the $\alpha_i$'s are rational then the moment polytope is a lattice polytope and the corresponding small cover is a non-singular, real toric variety. 
We refer the reader to \cite[Section 7]{MT} for more on the characteristic functions of non-singular toric varieties. 
\end{remark}

\begin{example}
Let $\alpha=(1,2,2,2)$ $\beta=(1,1,2,1)$ and $\gamma=(2,2,2,1)$ be generic length vectors. The shaded regions denotes the corresponding moment polytopes. The corresponding characterstic functions are given by matrices 
$\chi_{\alpha} = \begin{bmatrix} 
1 & 0 & 1 & 1\\
0 & 1 & 0 & 1
\end{bmatrix}$, $\chi_{\beta} = \begin{bmatrix} 
1 & 0 & 1 & 1 & 0\\
0 & 1 & 0 & 1 & 1
\end{bmatrix}$ and $\chi_{\gamma} = \begin{bmatrix} 
1 & 1 & 0 & 1 & 1 & 0\\
0 & 1 & 1 & 0 & 1 & 1
\end{bmatrix}.$

\begin{figure}[H]
\centering
$\includegraphics[scale=0.43]{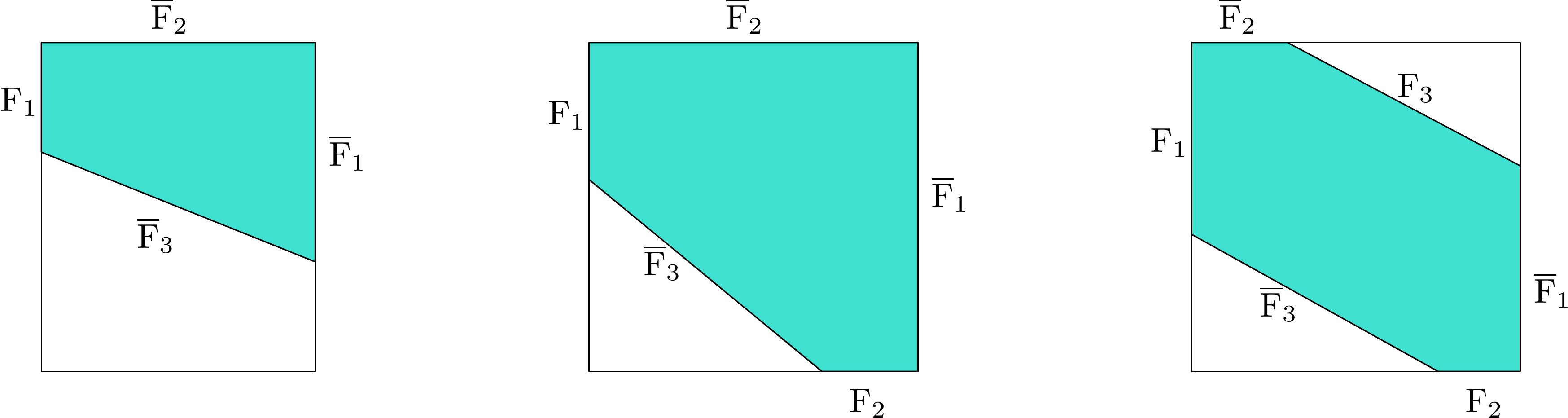}$
\caption{Moment polytopes} 
\end{figure}
\end{example}

\section{The proof of Theorem \ref{main}}\label{pmt}

A polytope is \emph{triangle-free} if it does not contain a triangular face of dimension $2$. 
It is easy to see that flagtopes are triangle-free. 
We state the important result due to G. Blind and R. Blind, which characterize those triangle-free polytopes that are also flagtopes. 
Recall that $I^k$ stands for the $k$-cube. 

\begin{theorem}\label{BLIND}\cite[Theorem 3]{GR}
If ${P}$ is a triangle-free convex polytope of dimension n then $f_{i}(P)\geq f_{i}(I^n)$ for $i=0,\dots, n$. 
In particular, such a polytope has at least $2n$ facets. 
Furthermore, if ${P}$ is simple then
\begin{enumerate}
\item $f_{n-1}(P)=2n$ implies that ${P}=I^n$ ;
\item $f_{n-1}(P)=2n+1$ implies that $P=P_{5}\times I^{n-2}$ where $P_5$ is a pentagon;
\item $f_{n-1}(P)=2n+2 $ implies that $P=P_{6} \times I^{n-2}$ or $P=Q\times I^{n-3}$ or $P = P_{5} \times P_{5} \times I^{n-4}$ where $P_{6}$ is a hexagon and $\mathrm{Q}$ is the $3$-polytope obtained from a pentagonal prism by truncating one of the edges forming a pentagonal facet.
\end{enumerate}
\end{theorem}

The above theorem helps us classify length vectors $\alpha$'s for which the moment polytope ${P}(\alpha)$ is a flagtope. 
 
\begin{theorem}\label{MAIN}
Let $\alpha$ be a generic length vector then the moment polytope $\mathrm{P}(\alpha)$ is a  flagtope if and only if the short code of $\alpha$ is one of the following:
\begin{enumerate}
	\item $<\{1,2,\dots,m-3,m\}>$
	\item $<\{1,2,\dots,m-4,m-2,m\}>$
	\item $<\{1,2,\dots,m-4,m-1,m\}>$
\end{enumerate}
\end{theorem}

\begin{proof}  
First we describe the idea of the proof, which is the same in each case. 
It follows from \Cref{FACET} that, the number of facets of $\mathrm{P}(\alpha)$ corresponding to each short code mentioned above, in that order, are $2(m-2)$, $2(m-2) + 1$ and $2(m-2) + 2$, respectively. 
Therefore, to show that the polytope $\mathrm{P}(\alpha)$ is triangle-free, we need to make sure that each $2$-dimensional face is not a $2$-simplex. 
Note that the $2$-dimensional faces of $\mathrm{P}(\alpha)$ correspond to admissible subsets of cardinality $m-4$. 
Therefore, in each case we determine these subsets and show that the corresponding $2$-dimensional faces have at least $4$ edges. 
Then we use \Cref{BLIND} to conclude $\mathrm{P}(\alpha)$ is a flagtope. 
We now analyze each short code. 
\smallskip
 
\noindent \textbf{Case 1.} \emph{The short code of $\alpha$ is $<\{1,2,\dots,m-3,m\}>$.}  
\vspace{1mm}

\noindent Let $J=J_{1}\cup J_{2}$ be an admissible subset of $[m-1]^{\pm}$ where $J_{1}\subset [m-1]^{+}$, $J_{2}\subset [m-1]^{-}$ with $J_{1}\cap -J_{2}=\emptyset$ and $|J|=m-4$. 
By \Cref{Add} we have that $J_{1}\cup\{m\}$ and $-J_{2}$ are $\alpha$-short subsets. 
Moreover $J_{1}\subset [m-3]$ as short code of $\alpha$ is $<\{1,2,\dots,m-3,m\}>$. 
Note that $\{m-2,m-1\}$ cannot be a subset of $-J_{2}$ since it is long. 
There are following three possibilities for $-J_{2}$.

\begin{enumerate}
\item We have \emph{$-J_{2}\subset [m-3]$}: 
Note that the facets of $F_{J}$ correspond to admissible subsets of cardinality $m-3$ containing $J$. 
We have, $\{i,m-2,m-1\}\subseteq (J_{1}\cup -J_{2})^c$ for $i\in [m-3]$. 
Clearly, the subsets $\{i\}\cup J_{1}\cup \{m\}$ and $\{i\}\cup -J_{2}$ are $\alpha$-short. 
Moreover, the subsets $\{m-2\}\cup -J_{2}$ and $\{m-1\}\cup -J_{2}$ are $\alpha$-short, since their complements contains long subsets,  $\{m-1,m\}$ and $\{m-2,m\}$ respectively. 
Therefore, the subsets $$\{i\}\cup J,\hspace{1mm} \{-i\}\cup J,\hspace{1mm} \{-(m-2)\}\cup J,\hspace{1mm} \{-(m-1)\}\cup J$$ are admissible. 
Note that these admissible subsets represents the facets of $F_{J}$. 
Clearly $F_{J}\cong \mathrm{I}^{2}$.

\item We have \emph{$m-2\in -J_{2}$ but $m-1\not\in -J_{2}$}: 
In this case we have $\{i,j,m-1\}\subseteq (J_{1}\cup -J_{2})^c$ where $\{i.j\}\subseteq [m-3]$. 
Note that the subsets, $$\{i\}\cup J_{1}\cup \{m\},\hspace{1mm} \{j\}\cup J_{1}\cup \{m\},\hspace{1mm} \{i\}\cup -J_{2},\hspace{1mm} \{j\}\cup -J_{2}$$ are $\alpha$-short. 
Therefore, the subsets $$\{i\}\cup J,\hspace{1mm} \{j\}\cup J,\hspace{1mm} \{-i\}\cup J,\hspace{1mm} \{-j\}\cup J$$ are admissible.  
Therefore, $2$-dimensional faces $F_{J}$ have at least four facets.

\item We have \emph{$m-1\in -J_{2}$ but $m-2\not\in -J_{2}$}: 
This is exactly similar to the earlier case.
\end{enumerate}

As a consequence of \Cref{FACET} we have, $f_{m-3}(\mathrm{P}(\alpha))=2(m-2)$.
Therefore, using \Cref{BLIND} we get $\mathrm{P}(\alpha)\cong I^{m-2}$, the $(m-2)$-cube.
\smallskip

\noindent \textbf{Case 2.} \emph{The short code of $\alpha$ is $<\{1,2,\dots,m-4,m-2,m\}>$.}
\vspace{1mm}

Note that $f_{m-3}(\mathrm{P}(\alpha))=2(m-2)+1$.
The short code information gives following possibilities for $J_{1}$ and $J_{2}$.
\smallskip
\begin{itemize}
\item \emph{Possibilities for $J_{1}$}: 
\begin{enumerate}
	\item $J_{1}\subset [m-3]$
	\item $m-2\in J_{1}$ and $J_{1}\setminus \{m-2\}\subset [m-4]$
\end{enumerate}
\item \emph{Possibilities for $J_{2}$}:
\begin{enumerate}
	\item  $-J_{2}\subset [m-2]$
	\item $-J_{2}=S \cup \{m-1\}$ where $S\subseteq [m-4]$ as $\{m-3,m-1\}$ and $\{m-2,m-1\}$ are long subsets.
\end{enumerate} 
\end{itemize}

\smallskip

\noindent We now consider all four possible combinations of $J_{1}$ and $J_{2}$.  
\smallskip

\begin{enumerate}
\item \emph{$J_{1}\subset [m-3]$ and $-J_{2}\subset [m-2]$} : 

Note that $\{i,j,m-1\}\subseteq (J_{1}\cup -J_{2})^c$ where $\{i.j\}\subseteq [m-2]$. Consider the following possibilities for $\{i.j\}$ :

\begin{enumerate}
\item $\{i,j\}\subset[m-3]$ : 
\vspace{1mm}

\noindent Observe that $S\cup \{m\}$ is $\alpha$-short if $S\subset [m-3]$. Therefore, $\{i\}\cup J_{1}\cup \{m\}$ and $\{j\}\cup J_{1}\cup \{m\}$ are $\alpha$-short subsets. 
Since $\{m-1,m\}$ is long, any subset of $[m-2]$ is $\alpha$-short.  Consequently, $\{i\}\cup-J_{2}$ and $\{j\}\cup -J_{2}$ are $\alpha$-short subsets. Hence  $$\{i\}\cup J,\hspace{1mm} \{j\}\cup J,\hspace{1mm} \{-i\}\cup J,\hspace{1mm} \{-j\}\cup J$$ are admissible subsets. Therefore, the $2$-dimensional faces $F_{J}$ have at least four edges. 
\smallskip

\item \emph{$i\in [m-3]$ and $j=m-2$} : 

In this case we have $\{i,m-2,m-1\}\subseteq (J_{1}\cup -J_{2})^c$.
Consider the following three sub-cases : 

\begin{enumerate}
\item \emph{If $ i\in [m-4]$ and $m-3\in J_{1}$ :}   
\vspace{1mm}

\noindent Note that $\{i\} \cup J_{1} \cup \{m\}$, $\{i\}\cup -J_{2}$ and $\{m-2\}\cup -J_{2}$  are short subsets. Since $m-3\in J_{1}$, $m-3\notin -J_{2}$. Therefore, $-J_{2}\subset [m-4]$. Hence, a subset $\{m-1\}\cup -J_{2}$ is $\alpha$-short. Therefore,  $$\{i\}\cup J,\hspace{1mm} \{-i\}\cup J,\hspace{1mm} \{-(m-2)\}\cup J,\hspace{1mm} \{-(m-1)\}\cup J$$ are admissible subsets. 
\smallskip

\item \emph{If $i\in[m-4]$ and $m-3\in -J_{2}$ :}  
This case is exactly same as above case.

\item \emph{If $i=m-3$ :}  
\vspace{1mm}

Note that $\{m-3,m-2,m-1\}\subseteq (J_{1}\cup -J_{2})^c$.
It is easy to see that the following subsets  $$\{m-3\}\cup J,\hspace{1mm} \{m-2\}\cup J,\hspace{1mm} \{-(m-3)\}\cup J,\hspace{1mm} \{-(m-2)\}\cup J,\hspace{1mm} \{-(m-1)\}\cup J$$ are admissible containing $J$ and represents the facets of $F_{J}$. Therefore, $2$-dimensional faces $F_{J}$ have exactly five edges. 
\end{enumerate}
\end{enumerate}
\smallskip

\item \emph{$J_{1}\subset [m-3]$ and $-J_{2}=S\cup \{m-1\}$ where $S\subset [m-4]$ :}  

In this case we have $\{i,j,m-2\}\subseteq (J_{1}\cup -J_{2})^c$ where $\{i.j\}\subseteq [m-3]$. Consider the following possibilities for $\{i.j\}$ :

\begin{enumerate}
\item \emph{Suppose $\{i.j\}\subset[m-4]$ :}  
\vspace{1mm}

\noindent It is easy to see that the following subsets $$\{i\}\cup J_{1}\cup \{m\},\hspace{1mm} \{j\}\cup J_{1},\hspace{1mm} \{j\}\cup -J_{2},\hspace{1mm} \{j\}\cup -J_{2}$$ are $\alpha$-short. Hence, the subsets $$\{i\}\cup J,\hspace{1mm} \{j\}\cup J,\hspace{1mm} \{-i\}\cup J,\hspace{1mm} \{-j\}\cup J$$ are  admissible. 
\smallskip

\item \emph{Suppose $i\in [m-4]$ and $j=m-3$ :}  

In this case we have $\{i,m-3,m-2\}\subseteq (J_{1}\cup -J_{2})^c$.
Since $m-3\notin J_{1}$, $J_{1}\subset [m-4]$. Hence the subsets, $\{m-3\}\cup J_{1}\cup \{m\}$ and $\{m-2\}\cup J_{1}\cup \{m\}$ are $\alpha$-short. Note that, $\{i\}\cup J_{1}$ and $\{i\}\cup -J_{2}$ are also $\alpha$-short subsets. Consequently, $$\{m-3\}\cup J,\hspace{1mm} \{m-2\}\cup J,\hspace{1mm} \{i\}\cup J,\hspace{1mm} \{-i\}\cup J$$  are admissible subsets. 
\end{enumerate}
\smallskip

\item \emph{$m-2\in J_{1}$, $J_{1}\setminus \{m-2\}\subset [m-4]$ and $-J_{2}\subset [m-2]$ :}  

Observe that, $\{i,j,m-1\}\subseteq (J_{1}\cup -J_{2})^c$ where $\{i.j\}\subseteq [m-3]$. Consider the following possibilities for $\{i.j\}$ :

\begin{enumerate}
\item \emph{If $\{i,j\}\subset [m-4]$ :}  
\vspace{1mm}

\noindent Clearly the subsets, $$\{i\}\cup J,\hspace{1mm} \{j\}\cup J,\hspace{1mm} \{-i\}\cup J,\hspace{1mm} \{-j\}\cup J$$ are admissible. 
\smallskip 

\item \emph{If $i\in [m-4]$ and $j= m-3$ :}  
\vspace{1mm}

In this case we have $\{i,m-3,m-1\}\subseteq (J_{1}\cup -J_{2})^c$ .
Since $m-2\in J_{1}$, $m-2\notin -J_{2}$. On the other hand, since  $m-3\notin -J_{2}$, $-J_{2}\subset [m-4]$. Note that, $S \cup \{m-1\}$ is $\alpha$-short for $S\subset [m-4]$. Therefore, $\{m-1\}\cup -J_{2}$ is $\alpha$-short subset. Consequently $\{-(m-1)\}\cup J$  is $\alpha$-admissible subset. Hence, $$\{i\}\cup J,\hspace{1mm} \{-i\}\cup J,\hspace{1mm} \{-(m-3)\}\cup J,\hspace{1mm} \{-(m-1)\}\cup J$$ are admissible subsets. 
\end{enumerate}
\smallskip

\item \emph{$m-2\in J_{1}$, $J_{1}\setminus \{m-2\}\subset [m-4]$ and $J_{2}=S\cup [m-1]$ , $S\subset [m-4]$ :}  

Here, we have $\{i,j,m-3\}\subseteq (J_{1}\cup -J_{2})^c$ where $\{i,j\}\subseteq [m-4]$. 
Clearly the subsets $$\{i\}\cup J,\hspace{1mm} \{-i\}\cup J,\hspace{1mm} \{j\}\cup J,\hspace{1mm} \{-j\}\cup J$$ are admissible subsets. 
\end{enumerate}
\smallskip
Finally we conclude that  $\mathrm{P}(\alpha)$ is  triangle-free if the short code of $\alpha$ is $<\{1,2,\dots,m-4,m-1,m\}>$. 
Moreover using Theorem \ref{BLIND} and Lemma \ref{FACET}, we get $\mathrm{P}(\alpha)\cong P_{5}\times I^{m-4}$, where $P_{5}$ is a pentagon.

\smallskip
\noindent \textbf{Case 3.} \emph{The short code of $\alpha$ is $<\{1,2,\dots,m-4,m-1,m\}>$.}
\vspace{1mm} 

Note that the short code of $\alpha$ gives us the following possibilities of $J_{1}$ and $J_{2}$ :
\smallskip

\begin{itemize}
\item \emph{Possibilities  for $J_{1}$:} 
\begin{enumerate}
	\item $J_{1}\subseteq [m-3]$
	\item $J_{1}\subseteq [m-4]$ and $m-2\in J_{1}$
	\item $J_{1}\subseteq [m-4]$ and $m-1\in J_{1}$
\end{enumerate}
\smallskip

\item \emph{Possibilities for $J_{2}$} 
\begin{enumerate}
    \item $-J_{2}\subseteq [m-3]$
    \item $-J_{2}\subseteq [m-4]$ and $m-2 \in -J_{2}$
    \item $-J_{2}\subseteq [m-4]$ and $m-1 \in -J_{2}$
\end{enumerate}
\end{itemize}

We consider the all nine possible combinations of $J_{1}$ and $J_{2}$. Let's begin with the fist combination:
\smallskip

\begin{enumerate}
\item \emph{$J_{1}\subseteq [m-3]$ and $-J_{2}\subseteq [m-3]$ :}  

Observe that, $\{i,m-2,m-1\}\subseteq (J_{1}\cup -J_{2})^c$ where $i\in [m-3]$.
Since $\{m-2,m-1\}$ is a long, the subsets $\{i\}\cup J_{1}\cup \{m\}$ and $i\cup -J_{2}$ are $\alpha$-short. Consider the following possibilities :

\begin{enumerate}
\item \emph{$m-3\in J_{1}$ :} 
\vspace{1mm}

Since $J_{1}\cap -J_{2}=\emptyset$, $-(m-3)\notin J_{2}$. Hence $-J_{2}\subseteq [m-4]$. Therefore, $\{m-2\}\cup -J_{2}$ and $\{m-1\}\cup -J_{2}$ are $\alpha$-short subsets. Consequently, the subsets $$\{i\}\cup J,\hspace{1mm} \{-i\}\cup J,\hspace{1mm} \{-(m-2)\}\cup J,\hspace{1mm} \{-(m-1)\}\cup J$$ are admissible. 
\smallskip 

\item \emph{$m-3\notin J_{1}$}:  
\vspace{1mm}

\noindent Therefore $J_{1}\subseteq [m-4]$. Hence, the subsets $\{m-2\} \cup J_{1}\cup \{m\}$ and $\{m-1\} \cup J_{1}\cup \{m\}$ are $\alpha$-short. Consequently, the subsets $$\{i\}\cup J,\hspace{1mm} \{-i\}\cup J,\hspace{1mm} \{m-2\}\cup J,\hspace{1mm}   \{m-1\}\cup J$$ are admissible. 
\end{enumerate} 
\smallskip

\item \emph{$J_{1}\subseteq [m-3]$ and $-J_{2}\subseteq [m-4]$, $m-2 \in -J_{2}$}:  

Since $\{m-3.m-2\}$ is long, $m-3\notin -J_{2}$. In this case we have $\{i,j,m-1\}\subseteq (J_{1}\cup -J_{2})^c$ where $\{i,j\}\subseteq [m-3]$. Therefore, the subsets $\{i\}\cup J_{1}\cup \{m\}$ and $\{j\}\cup J_{1}\cup \{m\}$ are $\alpha$-short. Consider the following possibilities:

\begin{enumerate}
\item \emph{$m-3\in J_{1}$}:  

Hence $\{i,j\}\subseteq [m-4]$. Therefore, $\{i\}\cup -J_{2}$ and $\{j\}\cup -J_{2}$ are $\alpha$-short subsets. Hence, the subsets $$\{i\}\cup J,\hspace{1mm} \{j\}\cup J,\hspace{1mm} \{-i\}\cup J,\hspace{1mm} \{-j\}\cup J$$ are admissible. 
\smallskip

\item \emph{$m-3\notin J_{1}$ :}  
\vspace{1mm}

\noindent Hence $J_{1}\subseteq [m-4]$. Therefore, $\{m-1\} \cup J_{1}\cup \{m\}$ is a $\alpha$-short subset. Since $m-3\notin -J_{2}$ ,  $\{i,m-3,m-1\}\subseteq (J_{1}\cup -J_{2})^c$ where $i\in [m-4]$. So $\{i\}\cup -J_{2}$ is $\alpha$-short subset. Therefore, the subsets, $$\{i\}\cup J,\hspace{1mm}  \{-i\}\cup J,\hspace{1mm} \{m-3\}\cup J,\hspace{1mm} \{m-1\}\cup J$$ are admissible. 
\end{enumerate}
\smallskip

\item \emph{$J_{1}\subseteq [m-3]$ and $-J_{2}\subseteq [m-4]$, $m-1 \in -J_{2}$ :} This case is exactly similar to the above case.

\item \emph{$J_{1}\subseteq [m-4]$, $m-2 \in J_{1}$ and $-J_{2}\subseteq [m-3]$ :}

In this case we have $\{i,j,m-1\}\subseteq (J_{1}\cup -J_{2})^c$ where $\{i,j\}\subseteq [m-3]$. Consider the following possibilities :

\begin{enumerate}
\item \emph{$m-3 \in -J_{2}$ :}  

Hence $\{i,j\}\subseteq [m-4]$. Therefore, the subsets, $$\{i\}\cup J_{1}\cup \{m\},\hspace{1mm} \{j\}\cup J_{1}\cup \{m\},\hspace{1mm} \{i\}\cup -J_{2},\hspace{1mm}  \{j\}\cup -J_{2}$$ are $\alpha$-short. Hence the subsets, $$\{i\}\cup J,\hspace{1mm} \{j\}\cup J,\hspace{1mm} \{-i\}\cup J,\hspace{1mm} \{-j\}\cup J$$ are admissible. 
\smallskip

\item \emph{$m-3\notin -J_{2}$ :}  

In this case we have $\{i,m-3,m-1\}\subseteq (J_{1}\cup -J_{2})^c$ where $i\in [m-4]$.
Therefore, $$\{i\} \cup J_{1}\cup \{m\},\hspace{1mm} \{i\}\cup -J_{2},\hspace{1mm} \{m-3\}\cup -J_{2},\hspace{1mm} \{m-1\}\cup -J_{2}$$ are $\alpha$-short subsets. Consequently, the subsets $$\{i\}\cup J,\hspace{1mm} \{-i\}\cup J,\hspace{1mm} \{-(m-3)\}\cup J,\hspace{1mm} \{-(m-1)\}\cup J$$ are admissible. 
\end{enumerate}
\smallskip

\item \emph{$J_{1}\subseteq [m-4]$, $m-2 \in J_{1}$ and $-J_{2}\subseteq [m-4]$ and $m-2 \in -J_{2}$ :}  
\vspace{1mm}

\noindent This case doesn't not arise as $J_{1}\cap -J_{2}\neq \emptyset$.
\smallskip

\item \emph{$J_{1}\subseteq [m-4]$, $m-2 \in J_{1}$ and $-J_{2}\subseteq [m-4]$, $m-1 \in -J_{2}$ :}  

In this case we have $\{i,j,m-3\}\subseteq (J_{1}\cup -J_{2})^c$ where $\{i.j\}\subseteq [m-4]$.
It is easy to see that the subsets,
$$\{i\}\cup J_{1}\cup \{m\},\hspace{1mm} \{j\}\cup J_{1}\cup \{m\},\hspace{1mm} \{i\}\cup -J_{2},\hspace{1mm} \{j\}\cup -J_{2}$$ are $\alpha$-short subsets. Hence the subsets $$\{i\}\cup J,\hspace{1mm} \{j\}\cup J,\hspace{1mm} \{-i\}\cup J,\hspace{1mm} \{-j\}\cup J$$ are admissible.  
\smallskip

\item \emph{$J_{1}\subseteq [m-4]$, $m-1 \in J_{1}$ and $-J_{2}\subseteq [m-3]$ :}  
This case is exactly similar to the case where, \emph{$J_{1}\subseteq [m-4]$, $m-2 \in J_{1}$ and $-J_{2}\subseteq [m-3]$}.

\item \emph{$J_{1}\subseteq [m-4]$, $m-1 \in J_{1}$ and $-J_{2}\subseteq [m-4]$, $m-2 \in -J_{2}$ :}  
This case is exactly similar to the case where, \emph{$J_{1}\subseteq [m-4]$, $m-2 \in J_{1}$ and $-J_{2}\subseteq [m-4]$, $m-1 \in -J_{2}$ :}

\item \emph{$J_{1}\subseteq [m-4]$ and $m-1 \in J_{1}$ and $-J_{2}\subseteq [m-4]$ and $m-1 \in -J_{2}$ }  
\noindent This case doesn't arise as $J_{1}\cap -J_{2}\neq \emptyset$. 
\end{enumerate}
Moreover, using \Cref{BLIND} and \Cref{FACET}, we get $\mathrm{P}(\alpha)\cong P_{6}\times I^{m-4}$, where $P_{6}$ is a hexagon. 

Now we prove the converse. Recall that, the \Cref{addsc} says that the poset of an admissible subsets determine the short code of a length vector. In particular, for generic length vectors $\alpha$ and $\beta$, if $\mathrm{P}(\alpha)\cong \mathrm{P}(\beta) $ then the short codes of $\alpha$ and $\beta$ coincides. Suppose, $\mathrm{P}(\alpha)$ is a flagtope for a generic length vector $\alpha$. Note that \Cref{FACET} give, $$2(m-2)\leq f_{m-3}(\mathrm{P}(\alpha))\leq 2(m-2)+2.$$ 
Therefore, there are only three possibilities for the number of facets of $\mathrm{P}(\alpha)$. 
Suppose that $f_{m-3}(\mathrm{P}(\alpha)) = 2(m-2)$. 
\Cref{BLIND} shows that $\mathrm{P}(\alpha)\cong I^{m-2}$, a $m-2$-cube. 
But we have shown that $I^{m-2}$ is the moment polytope for short code $<\{1,2,\dots,m-3,m\}>$. 
Therefore, the short code of $\alpha$ is $<\{1,2,\dots,m-3,m\}>$ whenever $\mathrm{P}(\alpha)$ is flagtope and $f_{m-3}(\mathrm{P}(\alpha)) = 2(m-2)$.
Similar arguments works when $f_{m-3}(\mathrm{P}(\alpha)) = 2(m-2)+1$ and $f_{m-3}(\mathrm{P}(\alpha)) = 2(m-2)+2$. 
\end{proof}

{The proof of \Cref{main} now follows from \Cref{MAIN} and \Cref{artv}.} 

\begin{example}
Below are some examples.
\begin{enumerate}
\item If $\alpha=(1,1,3,3,3)$ then the short code of $\alpha$ is $<\{1,2,5\}>$ and $\mathrm{P}(\alpha)\cong I^{3}$.
	
\begin{figure}[H]
	\centering
	$\includegraphics[scale=1]{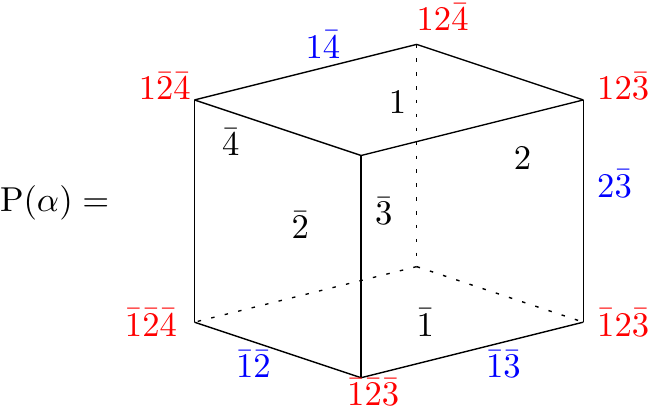}$
	\caption{}
\end{figure}

	\item If $\alpha=(1,2,2,5,3)$ then the short code of $\alpha$ is $<\{1,3,5\}>$ and $\mathrm{P}(\alpha)\cong P_{5}\times I$.

\begin{figure}[H]
	\centering
	$\includegraphics[scale=1]{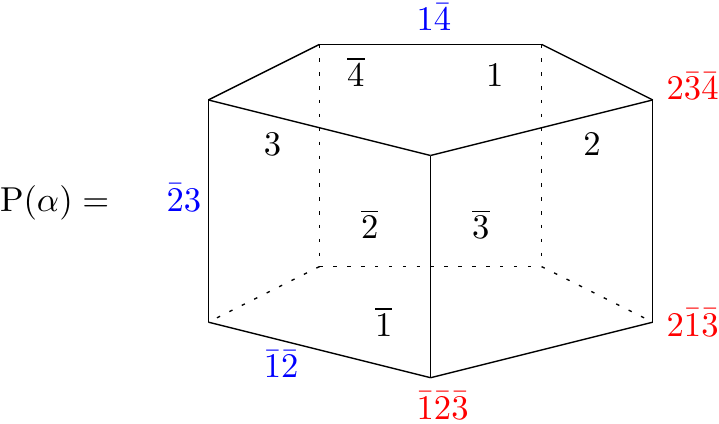}$
	\caption{}
\end{figure}
	
\item If $\alpha=(1,3,3,3,1)$ then the short code of $\alpha$ is $<\{1,4,5\}>$ and $\mathrm{P}(\alpha)\cong P_{6}\times I$.

\begin{figure}[H]
	\centering
	$\includegraphics[scale=1]{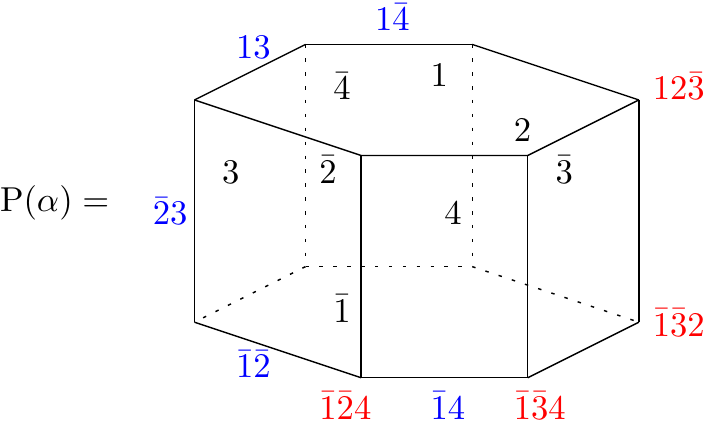}$
	\caption{}
\end{figure}	
\end{enumerate}
\end{example}

We now describe the homeomorphism type of these aspherical chain spaces. 
However, before that we mention an important and relevant result about polygon spaces. 
Let $<A_{1},\dots,A_{k}>$ be the genetic code of $\alpha=(\alpha_{2},\dots,\alpha_{m})$. 
Consider another generic length vector $\alpha^{+}$ whose genetic code is $<B_{1},\dots,B_{k}>$ where $B_{i}=\{a+1 : a\in A_{i}\}\cup \{1\}$. 
Note that $\alpha^{+}$ is an $m$-tuple. 
Hausmann \cite[Proposition 2.1]{hgdps} described a relationship between the polygon spaces $\overline{\mathrm{M}}_{\alpha^{+}}$ and $\overline{\mathrm{M}}_{\alpha}$.

\begin{proposition}\label{fib}
There polygon space $\overline{\mathrm{M}}_{\alpha^{+}}$ is diffeomorphic to the fibered product $S^{1}\times_{O(1)} \mathrm{M}_{\alpha}$, where $O(1)$ acts diagonally. 
\end{proposition}

Recall that $$q: \mathrm{M}_{\alpha}\longrightarrow \overline{\mathrm{M}}_{\alpha} $$ is a double cover. 
This double cover helps us to define a natural map $$\Phi : \overline{\mathrm{M}}_{\alpha^{+}}\longrightarrow \overline{\mathrm{M}}_{\alpha}.$$ 
It is easy to see that $\Phi$ is an $S^{1}$-fibration.

Now we return to aspherical chain spaces. 
Recall \Cref{cspps}; it says that given a generic length vector $\alpha=(\alpha_{1},\dots,\alpha_{m})$ we have an isomorphism between $$\mathrm{Ch}(\alpha)\longrightarrow \overline{\mathrm{M}}_{\alpha^{\prime}}$$ where 
$$\alpha^{\prime} = (\alpha_{1},\alpha_{2},\dots,\alpha_{m-1},\delta,\alpha_{m}+\delta)$$ for some positive real number $\delta > \sum_{i=1}^{m-1}\alpha_{i}$.

Suppose the short code of $\alpha$ is $<1,2,\dots,m-3,m>$. 
Then the genetic code of $\alpha^{\prime}$ is $<1,2,\dots,m-3,m+1>$. 
Let $\beta$ be a generic length vector whose genetic code is $<1,2,\dots,m-4,m>$. 
Clearly, $\beta^{+}=\alpha^{\prime}$. 
Let $\alpha(1)=(\alpha_{2},\dots,\alpha_{m}+\alpha_{1})$ be a generic length vector.
Note that the short code of $\alpha(1)$ is $<1,2,\dots, m-4,m-1>$ with $\alpha(1)^{\prime}=\beta$. We have an isomorphism $\mathrm{Ch}(\alpha(1))\longrightarrow \overline{\mathrm{M}}_{\alpha(1)^{\prime}}$. 
Using Hausmann's $S^{1}$-fibration described in the proof of Proposition \ref{fib}, we have the following maps.
$$\mathrm{Ch}(\alpha) \stackrel{\cong}{\longrightarrow} \overline{\mathrm{M}}_{\alpha^{\prime}} \stackrel{\Phi_{1}}{\longrightarrow} \overline{\mathrm{M}}_{\beta}\stackrel{\cong}{\longrightarrow} \mathrm{Ch}(\alpha(1)).$$ Clealry, the above composition of maps gives a $S^{1}$-fibration $\tilde{\Phi}$ from $$\mathrm{Ch}(\alpha)\stackrel{\tilde{\Phi}_{1}}{\longrightarrow} \mathrm{Ch}(\alpha(1)).$$ 

Let $I_{j}=\{1,2,\dots,j\}$.
By induction we get the chain of $S^{1}$-fibrations

$$\mathrm{Ch}(\alpha)\stackrel{\tilde{\Phi}_{1}}{\longrightarrow} \mathrm{Ch}(\alpha(I_{1}))\stackrel{\tilde{\Phi}_{2}}{\longrightarrow}\mathrm{Ch}(\alpha(I_{2}))\stackrel{\tilde{\Phi}_{3}}{\longrightarrow}\dots \stackrel{\tilde{\Phi}_{m-3}}{\longrightarrow} \mathrm{Ch}(\alpha(I_{m-3}))\stackrel{\tilde{\Phi}_{m-2}}{\longrightarrow} \{\star\},$$ where $\alpha(I_{j})=(\alpha_{j+1},\dots,\alpha_{m-1},\alpha_{m}+\displaystyle\sum_{i=1}^{j}\alpha_{i})$.

Recall that the moment polytope is the $(m-2)$-cube. 
Hence,the above chain of fibrations is a real Bott tower.
One can easily check that the characterstic matrix of $\mathrm{Ch}(\alpha)$ is
\[
\begin{bmatrix}
  \begin{matrix}
  \mathbf{I}_{m-2} & \rvline & \mathbf{I}'_{m-3} & \rvline & \mathbf{1}
  \end{matrix}
\end{bmatrix}
\]
where $\mathbf{I}_{m-2}$ is the block of $(m-2)\times (m-2)$ identity matrix, $\mathbf{I}'_{m-3}$ is the $(m-2)\times (m-3)$ block containing size $(m-3)$ identity matrix with the last row of zeros, the last block $\mathbf{1}$ is the column of $1$'s. 
The corresponding Bott matrix (that encodes the Stiefel-Whitney class of the fibration at each stage) is  

\[ 
\begin{bmatrix}
 \begin{matrix}
 {0} & \dots & 0 & {1} \\
 \vdots & \ddots & \vdots & \vdots\\
  0&\dots &0 & 1\\
 0 & \dots & 0 & 0
 \end{matrix}
\end{bmatrix}.
\]
Similarly, the remaining three aspherical chain spaces are towers of $S^1$-fibrations starting from a non-orientable surface. 

\bibliographystyle{plain}

\begin{thebibliography}{100}


\bibitem{GR}
G. Blind and R. Blind. 
\textit{Triangle-free polytopes with few facets.} 
Arch. Math. (Basel) 58 (1992), no. 6, 599–605.

\bibitem{MA}
M. R. Bridson, André Haefliger.
\textit{Metric Spaces of Non-Positive Curvature} Springer-Verlag.

\bibitem{VT}
V. M. Buchstaber, Taras E. Panov.
\textit{Toric Topology} 
Mathematical Surveys and Monographs
Volume 204, American Mathematical Society.

\bibitem{MTR}
M. Davis, T. Januszkiewicz and R. Scott .
\textit{Nonpositive curvature of blow-ups}.
Sel. math., New ser. 4 (1998) 491 – 547.
	
\bibitem{MT}
M. Davis and T. Januszkiewicz.
\textit{Convex polytopes, Coxeter Orbifolds and Torus Action}.Vol. 62, No. 2
Duke Mathematical Journal.


\bibitem{FS}
M. Farber, D.Schütz.
\textit{Homology of planar polygon spaces}  Geom. Dedicata 125, 75–92 (2007).


\bibitem{FAR}
M.Farber.
\textit{Invitation to Topological Robotics} EMS Zurich Lectures in Advanced Mathematics, 2008.

\bibitem{MG}
M. Gromov.
\textit{Hyperbolic groups in Essays in Group Theory} (S. M. Gersten, ed.).
M.S.R.I. Publ. 8. Springer-Verlag, New York, 1987.
	

\bibitem{JA}
J. C. Hausmann, Allen Knutson.
\textit{The cohomology ring of polygon spaces}
Annales de l’institut Fourier, tome 48, no 1 (1998), p. 281-321.

\bibitem{hkgrasmann}
J.-C. Hausmann, A. Knutson. 
\textit{Polygon spaces and Grassmannians}, Enseign. Math. (2) 43 (1997), 281–321.
    
\bibitem{hgdps}
J. C. Hausmann.
\textit{Geometric descriptions of polygon and chain spaces.}  
"Topology and Robotics", American Math. Soc. Contemporary Mathematics 438 (2007) 47-57. 
     
\bibitem{JE}
J. C. Hausmann and Eugenio Rodriguez.
\textit{The Space of Clouds in Euclidean Space} Experimental Mathematics, Vol. 13 (2004), No. 1.


\bibitem{MJ}
M. Kapovich, J. Millson.
\textit{On the moduli space of polygons in the Euclidean plane} J. Diff. Geom. 42,
430–464 (1995).

\bibitem{KM2}
M. Kapovich, J. Millson.
\textit{The symplectic geometry of polygons in Euclidean space} J. Diff. Geom. 44 (1996), 479-513.

\bibitem{GP}
G. Panina. 
\textit{Moduli space of a planar polygonal linkage: a combinatorial description}
Arnold Math. J. Vol. 3 No. 3 (2017), 351–364.

    
\bibitem{G}
G.M. Ziegler.
\textit{Lectures on Polytopes} Graduate Texts in Math., Springer-Verlag, Berlin, 1994.
\end{thebibliography}

\end{document}